\documentclass{amsart}
\usepackage{amssymb}
\usepackage{enumitem}
\usepackage{hyperref}

\usepackage{biblatex}
 \newtheorem{thm}{Theorem}[section]

 \newtheorem{lem}[thm]{Lemma}
 \newtheorem{prop}[thm]{Proposition}
 \theoremstyle{definition}
 \newtheorem{defn}[thm]{Definition}
 \theoremstyle{remark}
 
 \newtheorem{ex}[thm]{Example}
 \numberwithin{equation}{section}

\usepackage{enumitem}
\usepackage{mathtools}
\DeclareMathOperator*{\esssup}{ess\,sup}
\DeclareMathOperator*{\essinf}{ess\,inf}

\begin{document}
\title[Noncommutative Down Spaces]{Noncommutative Down Spaces}
\author[Alejandro {Santacruz Hidalgo}]{Alejandro {Santacruz Hidalgo}}

\address{%
Pozna\'n University of Technology \\ 
Institute of Mathematics \\
Piotrowo 3A, 60-965 Pozna\'n\\
Poland}

\email{asantacr@uwo.ca}


\date{June 3, 2026}

\begin{abstract}
Given a von Neumann algebra with a semifinite normal faithful trace, the notion of a core decreasing operator is introduced. This construction depends on a family of commuting projections called an ordered core. With these tools in place, a version of the level function of a measurable operator is introduced. The down space of a strongly symmetric space of $\tau$-measurable operators is defined.
\end{abstract}

\maketitle

\section{Introduction}\label{introduction} 
A theory of noncommutative integration was started in the early fifties by Segal \cite{Segal53}, laying the foundations for noncommutative Lebesgue spaces. Here, a von Neumann algebra takes the role of measure space, and a semifinite trace is the substitute for the integral. In \cite{ovcinnikov70}, Ov\v cinnikov extended interpolation results of rearrangement invariant spaces due to Calder\'on \cite{calderon} to this noncommutative setting. A more general method to construct noncommutative symmetric spaces was given in \cite{dodds89}, where the notion of $\tau$-measurable operators (due to Nelson \cite{nelson74}) was used. Since then, several properties of symmetric spaces of $\tau$-measurable operators have been studied. For instance, a theory of K\"othe duality in \cite{dodskothe} or geometric properties of these spaces in \cite{kaminska17}.

In the classical integration theory on a measure space, there is a family of function spaces closely related to rearrangement invariant spaces, called the down space. In the case of the $L^{p}$ spaces over $[0,\infty)$ with the Lebesgue measure, the down space $(L^p)^o$ is given by the restricted K\"othe dual
$$
\|f\|_{(L^p)^o} = \sup \left\{ \int_0^\infty f \left|g\right|: \|g\|_{L^{p'}} \leq 1 \text{ and } g \text{ is nonincreasing}  \right\} = \|f^o\|_{L^{p}}, 
$$
here $1/p + 1/p' = 1$ and $f^o$ is a nonnegative nonincreasing function called \textit{the level function of} $f$. For locally integrable functions, $f^o$ is defined by requiring that the function $s \mapsto \int_{0}^s f^o$ is the least concave majorant of the function $s \mapsto \int_0^s |f|$. The construction of the level function is due to Halperin \cite{halperin}, with an independent proof by Lorentz \cite{lorentz}. The level function provides an improvement of H\"older's inequality in the presence of monotonicity since the inequality $\int_0^\infty |f|g \leq \|f^o\|_{L^p}\|g\|_{L^{p'}} $ holds whenever $g$ is nonincreasing. 

Since then, the level function and the down space construction have been extended to general measures over the real numbers and their properties studied extensively (see \cite{sinnamon94}, \cite{sinnamon01},\cite{mastylo},\cite{mastylo06}). The level function has been used to give formulas for the dual spaces of Lorentz and Orlicz-Lorentz spaces in \cite{kaminsma19} and \cite{tag24}, to prove weighted Hardy and Fourier inequalities \cite{sinnamonfourier03}, to transfer monotonicity (from kernel to weight) in weighted norm inequalities for general positive integral operators \cite{sinnamon03}, and to provide equivalent norms for traditional and abstract Ces\`aro spaces that facilitate interpolation of these spaces and of their duals \cite{lesnik2016}.

Recently, in \cite{coredecreasing}, the level function and the down space construction were extended to general measure spaces by developing a theory of generalized decreasing functions. These decreasing functions rely on the notion of a measure space with an ordered core, introduced in \cite{sinnamon22} to study abstract Hardy operators. The flexibility provided by choosing a suitable ordered core was applied in \cite{santacruz2024} to study endpoint cases of the abstract Hardy's inequality and to provide new proofs for necessary and sufficient conditions for Hardy's inequality to hold in a metric measure space. 

The purpose of this paper is to provide a theory of down spaces for $\tau$-measurable operators and to extend the level function construction and generalized decreasing functions to this noncommutative setting. In Section 2, we set out some necessary background and fix the notation used throughout the paper. In Section 3, we introduce our main tool, the ordered core, to encode the monotonicity properties of our spaces and study some of their basic properties. In Section 4, we introduce the notion of morphism between ordered cores and prove that each morphism induces a linear map between the von Neumann algebras. Section 5 describes a functional connection between core decreasing operators and nonnegative nonincreasing functions in the half line. Our main results are in Section 6, where we construct a noncommutative level function to define noncommutative down spaces.

\section{Notation and Background}\label{background}
Let $H$ be a complex Hilbert space equipped with an inner product $\langle , \rangle$, $B(H)$ is the space of bounded linear operators of $H$, and its identity element will be denoted $1_{B(H)}$. For a nonempty subset $A$ of $B(H)$, the commutant of $A$ is denoted by $A'$ is defined as $A' = \{ x \in B(H): xy = yx, \, \forall y \in A  \}$. A von Neumann algebra $\mathcal{M}$ is a $*$-subalgebra of $B(H)$ such that $\mathcal{M}'' = \mathcal{M}$. If $A$ is a $*$-subalgebra of $B(H)$, then $A''$ is the von Neumann algebra generated by $A$. A von Neumann algebra $\mathcal{M}$ is \textit{abelian} if $\mathcal{M} \subseteq \mathcal{M}'$.

A closed linear operator $x: \mathfrak{D}(x) \to H$ defined on a dense linear subspace of $H$, is called self adjoint if $x^* = x$, normal if $x^*x = xx^*$, and unitary if $xx^*=1_{B(H)} = x^*x$. A self adjoint operator is positive if $\langle x \xi,\xi\rangle \geq 0$ for all $\xi \in \mathfrak{D}(x)$. If $A$ is a family of closed densely defined operators, by $A_h$ we mean its collection of self adjoint operators, by $A^+$ we mean its collection of positive operators, and by $U(A)$ its collection of unitary operators. If $x,y$ are positive densely defined closed operators, we say that $x \leq y$ if $\mathfrak{D}(y) \subseteq \mathfrak{D}(x)$ and $\langle x \xi , \xi\rangle \leq \langle y \xi , \xi\rangle$ for all $\xi \in \mathfrak{D}(y)$.

For a partially ordered set $(R,\leq)$ and $\{x_\alpha\}$ an increasing net in $R$ such that $\sup_{\alpha} x_{\alpha}$ exists in $R$, we write $x_\alpha \uparrow x$. We write $x_\alpha \downarrow x$ if $\{x_\alpha\}$ is a decreasing net and $x = \inf_{\alpha} x_{\alpha}$.

Given a von Neumann algebra $\mathcal{M}$, we denote the set of projections by $P(\mathcal{M})$, that is, $p \in \mathcal{M}$ such that $p^* = p = p^2$. For $p,q \in P(\mathcal{M})$, we say that $p \leq q$ if $p(H) \subseteq q(H)$ or, equivalently, if $pq = p = qp$. For an increasing net of projections $\{p_{\alpha}\} \in \mathcal{M}$ and a decreasing net of projections $\{q_{\alpha}\} \in \mathcal{M}$ we define $\vee p_{\alpha}$ to be the projection onto $\overline{\cup_{\alpha} p_{\alpha}(H)}$ and $\wedge q_{\alpha}$ to be the projection onto $\cap_{\alpha} p_{\alpha}(H)$. It is well known that $p_{\alpha} \uparrow \vee p_{\alpha} \in P(\mathcal{M})$ and $q_{\alpha} \downarrow q \in P(\mathcal{M})$. An operator $v \in B(H)$ is a partial isometry if $\|v\xi\|_{H} = \|\xi\|_{H}$ for all $\xi \in \ker(v)^\perp$ or equivalently if $v^*v$ is a projection. We write $p \sim q$ to mean that two projections are equivalent, that is, if there exists a partial isometry $v \in \mathcal{M}$ such that $p = v^*v$ and $q = vv^*$.

A mapping $\tau:\mathcal{M}^+ \to [0,\infty]$ is a \textit{trace} if for all $x,y \in \mathcal{M}^+$, all $\alpha \in [0,\infty)$, and all $u \in U(\mathcal{M})$ we have that $\tau(x+y) = \tau(x) + \tau(y)$, $\tau(\alpha x) = \alpha \tau(x)$ and $\tau(u^*xu) = \tau(x)$. In addition, 
\begin{itemize}
    \item If $\tau(x) = 0$ implies $x = 0$, we say that the trace is \textit{faithful}.
    \item If $\tau(x) = \sup_{\alpha} \tau(x_n)$ whenever $x_n \uparrow x$, we say that the trace is \textit{normal}.
    \item If for each $x$, such that $\tau(x) > 0$, there exists $z \in \mathcal{M}^+$ such that $z \leq x$ and $0 < \tau(z) < \infty$, we say that $\tau$ is \textit{semifinite}.
\end{itemize}

If the von Neumann algebra is abelian, the following result \cite[Theorem~7.21]{abelianVN} will be useful.
\begin{thm}\label{abeliano}
    If $(\mathcal{N},\nu)$ is an abelian von Neumann algebra with a semifinite normal faithful trace $\nu$. Then, there exists a localizable measure space $(X,\Sigma,\theta)$ and a $*$-isomorphism $\Phi: \mathcal{N} \to L^{\infty}(X,\Sigma,\theta)$ such that $\tau(x) = \int_{X} \Phi(x) \, d\theta$ for all $x \in \mathcal{N}$. 
\end{thm}

A closed densely defined operator $x: \mathfrak{D}(x) \to H$ is affiliated with the von Neumann algebra $\mathcal{M}$ if $u(\mathfrak{D}(x)) \subseteq \mathfrak{D}(x)$ for all $u \in U(\mathcal{M'})$, in this case, we write $x \eta\mathcal{M}$. We say that a closed densely defined operator affiliated with $\mathcal{M}$ is \textit{measurable} if there exists a sequence of projections $\{p_n\} \in P(\mathcal{M})$ such that $p_n \uparrow 1$, $p_n(H) \subseteq \mathfrak{D}(x)$, and for all $n \in \mathbb{N}$, if $q \in P(\mathcal{M})$ such that $q \leq p_n^\perp$ and $q \sim p_n^\perp$, it follows that $q = p_n^\perp$. The sequence $\{p_n\}$ is called a \textit{determining sequence} of $x$. By $S(\mathcal{M})$, we mean the collection of measurable operators with respect to $\mathcal{M}$, which is a $*$-algebra under the closure of the operators for sum and product. 

For an operator $x \in S(\mathcal{M})_h$, we denote by $e^x(\cdot)$ its spectral measure, that is, the unique projection valued measure from $\text{Borel}(\mathbb{R}) \to P(\mathcal{M})$ such that $x = \int_{\mathbb{R}} r \, de^x(r)$. The positive and negative parts of $x$ are given by $x^+ = \int_{[0,\infty)} r \, de^x(r)$ and $x^- = -\int_{(-\infty,0]} r \, de^{x}(r)$, with $x = x^+ - x^-$. If $f:[0,\infty) \to \mathbb{C}$ is a Borel measurable function, bounded on compact sets, then $f(x)$ is defined as $\int_{\mathbb{R}} f(r) \, de^{x}(r)$ and $f(x) \in S(\mathcal{M})$ (See \cite[Proposition~2.2.10]{dodsbook}).

If $x \in S(\mathcal{M})$, the operator $|x|$ is defined as the positive square root of $x^*x$, it characterized as the unique positive operator $|x| \in S(\mathcal{M})^+$ such that $\mathfrak{D}(x) = \mathfrak{D}(|x|)$ and $\|x\xi\|_{H} = \||x| \xi \|_{H}$ for all $\xi \in \mathfrak{D}(x)$. Also, there exists a partial isometry $v \in \mathcal{M}$ such that $x = v|x|$, $|x| = v^*x$ and $x^* = v^*|x^*|$, this factorization is called the \textit{polar decomposition} of $x$.

If $\mathcal{M}$ is a von Neumann algebra, equipped with a semifinite normal faithful trace $\tau$, and $x$ is a closed operator affiliated with $\mathcal{M}$, we say that $x$ is $\tau$-measurable if for every $\epsilon > 0$ there exists $p \in P(\mathcal{M})$ such that $p(H) \subseteq \mathfrak{D}(x)$ and $\tau(p^\perp) < \epsilon$. By $S(\tau)$ we denote the $*$-algebra of $\tau$-measurable operators. It follows from the definitions that $\mathcal{M} \subseteq S(\tau) \subseteq S(\mathcal{M})$. The sets
$$
V(\epsilon,\delta) = \left\{ x \in S(\tau): \tau\big( e^{|x|}(\epsilon,\infty)\big)  < \delta  \right\}, \quad \epsilon,\delta > 0,
$$
define a neighbourhood base at zero for a metrizable Hausdorff topology called the \textit{measure topology} $\mathcal{T}$ on $S(\tau)$. Under the measure topology, $S(\tau)$ is a complete topological $*$-algebra with $\mathcal{M}$ as a dense subset. A sequence $x_n \to 0$ in the measure topology if $\tau\big(e^{|x_n|}(\epsilon,\infty) \big) \to 0$ as $n \to \infty$. 

If $x \in S(\tau)$, its \textit{distribution function} is defined by
$$
d(s;x) = \tau\big( e^{|x|}(s,\infty)\big), \quad s \geq 0,
$$
and its \textit{singular value function} is defined by
$$
\mu(t;x) = \inf \left\{ s \geq 0: d(s;x) \leq t \right\}.
$$
The trace can be extended to $S(\tau)^+$ by the formula $\tau(x) = \int_{0}^\infty \mu(\cdot,x)$ and the extended trace remains semifinite, normal, and faithful. The space $L^{1}(\tau)$ is the collection $\{ x \in S(\tau): \tau(|x|) < \infty \}$ with the norm $\|x\|_{L^{1}(\tau)} = \tau(|x|)$. The trace extends uniquely to a positive linear functional on $L^{1}(\tau)$ satisfying $\tau(xy) = \tau(yx)$ for all $y \in \mathcal{M}$.  

A Banach space $E \subseteq S(\tau)$, equipped with a norm $\|\cdot\|_{E}$ is a \textit{strongly symmetric space} of $\tau$-measurable operators if for any $x \in S(\tau)$ and $y \in E$ such that 
$$
\int_{0}^s \mu(\cdot \ ; x) \leq \int_{0}^s \mu(\cdot \ ; y), \quad \forall s \geq 0,
$$
we have that $x \in E$ and $\|x\|_{E} \leq \|y\|_{E}$. 
If $E$ is a strongly symmetric space, then $\|uxv\|_{E} \leq \|u\|_{B(H)}\|v\|_{B(H)} \|x\|_{E}$ for every $u,v \in \mathcal{M}$. Moreover, $\|x\|_E = \||x|\|_{E} = \|x^*\|_{E} = \||x^*|\|_{E}$.

For a strongly symmetric space $E$, its K\"othe dual is defined by the seminorm
$$
\|x\|_{E^\times} = \sup \left\{ \tau(|yx|): \|y\|_{E} \leq 1 \right\}.
$$
The collection $E^\times$ is $\{ x \in S(\tau): \|x\|_{E^\times} \} < \infty$. The spaces $L^{1}(\tau)$ and $(\mathcal{M},\|\cdot\|_{B(H)})$ are strongly symmetric spaces, moreover $L^{1}(\tau) = \mathcal{M}^\times$ and $L^{1}(\tau)^{\times \times} = L^{1}(\tau)$.


See \cite{dodsbook} for more details.

\section{Ordered cores}\label{cores}
Throughout this section, suppose that $\mathcal{M}$ is a von Neumann algebra, equipped with a semifinite normal faithful trace $\tau$. We study a subset of the projections of a von Neumann algebra called an ordered core. This object will encode the monotonicity properties of our space and allows us to extend the definition of decreasing functions to the von Neumann algebra setting. In the commutative case, a variation of this construction was used in \cite{sinnamon22},\cite{coredecreasing}, and \cite{santacruz2024}.

\begin{defn}\label{coredefinition}
A family $\mathcal{A} = \{p_i\}_{i \in I} \in P(\mathcal{M})$ is an \textbf{ordered core} provided:
\begin{enumerate}[label=(\roman*)]
    \item \label{zerocore}The zero projection belongs to $\mathcal{A}$.
    \item \label{totalorder} For each $i,j \in I$, $p_i \leq p_j$ or $p_j \leq p_i$.
    \item \label{finitetrace} For each $i \in I$, $\tau(p_i) < \infty$.
\end{enumerate}
Let $H_\mathcal{A} = \overline{\cup_{i \in I} p_i(H)}$. We say that the ordered core is \textbf{$\sigma$-bounded} if there exists a countable subset $J \subseteq I$ such that $H_\mathcal{A} = \overline{\cup_{j \in J} p_j(H)}$. We say that the ordered core is \textbf{full} if $H_\mathcal{A} = H$. 
\end{defn}


We begin collecting some simple properties of ordered cores and the normal faithful trace $\tau$.
\begin{prop}\label{coreprop1}
    Let $\mathcal{M}$ be a von Neumann algebra, equipped with a semifinite normal faithful trace $\tau$, and an ordered core $\mathcal{A}$. 
    \begin{enumerate}[label=(\roman*)]
        \item \label{faithfulness} If $p,q \in \mathcal{A}$ and $\tau(p) = \tau(q)$, then $p = q$.
        \item \label{uniquenessincreasing} If $\{p_n\},\{q_n\}$ are increasing sequences in $\mathcal{A}$ such that $p_n \uparrow p$ and $q_n \uparrow q$ for some $p,q \in P(\mathcal{M})$ such that $\tau(p) = \tau(q)$, then $p = q$.
        \item \label{existencedecreasing} If $\{q_{n}\}$ a decreasing sequence $\mathcal{A}$. Then, $q = \wedge \{q_n\} \in P(\mathcal{M})$, and $\tau(q_{n}) \downarrow \tau(q)$.
        \item \label{uniquenessdecreasing} If $\{p_n\},\{q_n\}$ are decreasing sequences in $\mathcal{A}$ such that $p_n \downarrow p$ and $q_n \downarrow q$ for some $p,q \in P(\mathcal{M})$ such that $\tau(p) = \tau(q)$, then $p = q$.  
        \item \label{commutetrace} Let $p,q \in \mathcal{A}$ such that $q-p \geq 0$ and $x \in \mathcal{M}^+$ then $(q-p)x(q-p) \in \mathcal{M}^+$ and 
        \begin{equation}\label{commuteboundeq}
            \tau\big((q-p)x\big) = \tau\big((q-p)x(q-p)\big) = \tau\big(x(q-p)\big) \leq \|x\|_{B(H)} \tau(q-p).
        \end{equation}
        \item \label{orderpcore} If $x,y \in S(\mathcal{M})^+$ and $p \in \mathcal{A}$ such that $x \leq y$, then $\tau(xp) = \tau(pxp) = \tau(px)$ and $\tau(xp) \leq \tau(yp)$.
    \end{enumerate}
    
\end{prop}
\begin{proof}
By the total order of $\mathcal{A}$, we may assume that $p \leq q$, this means that $q-p \geq 0$. Since $\tau(p-q) = \tau(p) - \tau(q) = 0$, we conclude that $y-x$ is a positive element with zero trace. We conclude that $y = x$ by the faithfulness of the trace and prove \ref{faithfulness}.

We now prove \ref{uniquenessincreasing}. By normality of the trace $\tau(p_n) \uparrow \tau(p)$ and $\tau(q_n) \uparrow \tau(q)$. If $\tau(p) = \infty$, then by item \ref{finitetrace} in \ref{coredefinition} we have that $p$ and $q$ must be the orthogonal projection onto $H_{\mathcal{A}}$. Therefore, we may assume that the projections $p$ and $q$ have finite trace. Let $\xi \in p(H) \setminus q(H)$, we must have that $p_n(\xi) \to \xi$ and $q_n(\xi) \not\to \xi$, with convergence in $H$. 

There exists some $N \in \mathbb{N}$ such that for all $n > N$, we must have that $p_n(\xi) \not\in q_n(H)$, otherwise we can extract subsequences $\{p_{n_j}\}$ and $\{q_{n_j}\}$ such that $p_{n_j}(\xi) \in q_{n_j}(H) \subseteq q(H)$. Since $p_{n_j} \to \xi$ and $q(H)$ is closed, then $\xi \in q(H)$ contradicting the choice of $\xi$. By the total order of $\mathcal{A}$, we conclude that $q_{n_j} \leq p_{n_j}$. Since $q_{n_j} \uparrow q$ and $p_{n_j} \uparrow p$, we get $q \leq p$. So $p-q$ is a positive element with zero trace; it follows that $p = q$, proving \ref{uniquenessincreasing}.  

To prove \ref{existencedecreasing}, let $\{q_n\}$ be a decreasing sequence in $\mathcal{A}$. Since $P(\mathcal{M})$ is a complete lattice, $q = \wedge \{q_n\} \in P(\mathcal{M})$. It is clear that the sequence $\{ q_n - q\}$ is decreasing, bounded above by $q_1 \in L^1(\tau)$ and $q_n - q \downarrow 0$. An application of \cite[Lemma~2.6.2]{dodsbook} shows that $q_n -q \to 0$ in the measure topology. By \cite[Corollary~2.5.8]{dodsbook} $\tau(q_n -q) \to 0$, thus $\tau(q_n) \downarrow \tau(q)$ and proves \ref{existencedecreasing}. 

To prove \ref{uniquenessdecreasing}, let $\{q_n\},\{p_n\}$ be decreasing sequences in $\mathcal{A}$ such that $p_n \downarrow p$, $q_n \downarrow q$, $\tau(p) = \tau(q)$. Suppose that $p \not= q$, seeking a contradiction.

By \ref{existencedecreasing} and the uniqueness of infima in $P(\mathcal{M})$, we get that $p = \wedge\{p_n\}$, $q = \wedge\{q_n\}$,  $\tau(p_n) \downarrow \tau(p)$ and $\tau(q_n) \downarrow \tau(q)$. If $\xi \in q(H) \setminus p(H)$, we may extract subsequences $\{p_{n_j}\},\{q_{n_j}\}$ such that $\xi \not\in p_{n_j}(H)$ and $\xi \in q_{n_j}(H)$. By total order, $p_{n_j} \leq q_{n_j}$. Taking infimum, we get that $p \leq q$. Therefore $q-p$ is a positive element with zero trace, it follows that $p = q$, which is a contradiction. Hence, $q(H) \subseteq p(H)$, hence $q \leq p$, by the faithfulness of $\tau$, we get the same contradiciton $q=p$. This proves \ref{uniquenessdecreasing}.

To prove \ref{commutetrace} let $p,q \in \mathcal{A}$ and $x \in \mathcal{M}^+$ then 
$$
\big((q-p)x(q-p)\big)^* = (q-p)^* x^* (q-p)* = (q-p) x (q-p),
$$
so $(q-p)x(q-p)$ is self adjoint. To show it is positive, let $\xi \in H$. Since $p<q$, we have that $ r= q-p$ is a projection. Hence,
\begin{align*}
    \langle rxr \xi , \xi \rangle &= \langle xr \xi , r\xi \rangle \leq \|x r \xi \|_{H} \|r \xi\|_{H} \leq \|x\|_{B(H)} \|r \xi\|_{H}^2 = \|x\|_{B(H)} \langle r \xi , r \xi \rangle\\
    &= \|x\|_{B(H)} \langle r \xi , \xi \rangle = \langle \|x\|_{B(H)} r \xi , \xi \rangle.  
\end{align*}
It follows that $rxr \leq \|x\|_{B(H)} r$, therefore $\tau(rxr) \leq \|x\|_{B(H)} \tau(r)$.

Notice that $\tau(q-p) \leq 2 \tau(q) < \infty$, therefore \cite[Theorem~1.15.8]{dodsbook} shows that $\tau\big((q-p)x\big) = \tau(x (q-p))$. For the final equality, we use the fact that $(q-p)^2 = q-p$ to get
$$
\tau\big((q-p)x\big) = \tau\big((q-p)(q-p)x\big) = \tau\big((q-p)x(q-p)\big).  
$$

To prove \ref{orderpcore}, notice that, for each $p \in \mathcal{A}$, we have that $p \in L^{1}(\tau)$. Therefore, $xp, px, pxp$ belong to $S(\tau)$. By \cite[Proposition~3.4.32]{dodsbook} we have
$$
\tau(xp) = \tau\big((xp)p\big) = \tau(p^{1/2}xp p^{1/2}) = \tau(pxp), 
$$
and
$$
\tau(px) = \tau\big(p(px)\big) = \tau(p^{1/2}pxp^{1/2}) = \tau(pxp).
$$

Finally, if $x \leq y$, it follows that $pxp \leq pyp$. Therefore, $\tau(pxp) \leq \tau(pyp)$. By the previous observation, we get that $\tau(xp) \leq \tau(yp)$ to complete the proof of \ref{orderpcore}.

\end{proof}

For any ordered core we associate the set of all tracial values of its elements; this set will play a crucial role in the sequel. 
$$
\Gamma_{\mathcal{A}} = \{ x \in [0,\infty): x = \tau(p) \text{ for some } p \in \mathcal{A} \}.
$$

\begin{defn}
    An ordered core $\mathcal{A}$ is said to be maximal if $\Gamma_{\mathcal{A}}$ is a closed subset of $[0,\infty)$. 
\end{defn}

We now show that an ordered core admits a unique maximal extension.

\begin{thm}\label{maximal}
    Let $\mathcal{A}$ be an ordered core, then there exists a unique ordered core $\mathcal{A}_M$ such that $\mathcal{A} \subseteq \mathcal{A}_M$ and $\Gamma_{\mathcal{A}_M}$ is the closure of $\Gamma_{\mathcal{A}}$ in $[0,\infty)$. 
\end{thm}
\begin{proof}
    \textbf{Existence:} If $\Gamma_{\mathcal{A}}$ is closed, then we can set $\mathcal{A}_M = \mathcal{A}$ and there is nothing to prove. Otherwise define
    \begin{align*}
S &= \{s \in \overline{\Gamma_{\mathcal{A}}} \setminus \Gamma_{\mathcal{A}}: \text{there exists } \{p_n\} \in \mathcal{A} \text{ such that } \tau(p_n) \uparrow s \},\\
I &= \{s \in \overline{\Gamma_{\mathcal{A}}} \setminus (\Gamma_{\mathcal{A}} \cup S): \text{there exists } \{q_n\} \in \mathcal{A} \text{ such that } \tau(q_n) \downarrow s \}.
    \end{align*}
It is clear that $\overline{\Gamma_A} = \Gamma_A \cup S \cup I$ and that the union is disjoint.
    
    Let $s \in S$, then there exists a sequence $p_{n} \in \mathcal{A}$ such that $\tau(p_n) \uparrow s$ and $\{p_n\}$ is an increasing net in $P(\mathcal{M})$, therefore there exists $p_s \in P(\mathcal{M})$ such that $p_n \uparrow p_s$. By normality of the trace $\tau(p_n) \uparrow \tau(p_s)$. 
    
    Let $s \in I$, then there exists a sequence $\{q_n\} \in \mathcal{A}$ such that $\tau(q_n) \downarrow s$ and  $\{q_n\}$ is an decreasing net in $P(\mathcal{M})$, therefore there exists $q_s \in P(\mathcal{M})$ such that $q_n \downarrow q_s$. By item \ref{existencedecreasing} in Proposition \ref{coreprop1} we get that $\tau(q_n) \downarrow \tau(q_s)$.

    Define
    $$
\mathcal{A}_M = \mathcal{A} \cup \{p_s: s \in S \} \cup \{ q_s: s \in I\}.
    $$
    It is immediate that $\Gamma_{\mathcal{A}_M} = \overline{\Gamma_A}$, it remains to show that $\mathcal{A}_M$ is an ordered core. Clearly $0 \in \mathcal{A} \subseteq \mathcal{A}_M$ and by construction every element in $\mathcal{A}_M$ has finite trace. It remains to show that $\mathcal{A}_M$ is totally ordered.
    
    Let $x$ and $y$ be distinct elements in $\mathcal{A}_M$. If $\tau(x) \not= \tau(y)$, without loss of generality we may assume that $\tau(x) < \tau(y)$. Since $\tau(y) - \tau(x) > 0$, there exist monotone (or constant) sequences $\{p_n\}$ and $\{q_m\}$ in $\mathcal{A}$ such that $\tau(p_n) \to \tau(x)$, $\tau(q_m) \to \tau(y)$ and $\tau(p_n) < \tau(q_m)$ for each $n,m \in \mathbb{N}$. By the total order of $\mathcal{A}$, it follows that $p_n \leq q_m$ for each $n,m \in \mathbb{N}$. Since $x$ is $\vee p_n$ or $\wedge p_n$, we get $x \leq q_m$ for each $m \in \mathbb{N}$. Since $y$ is $\vee q_m$ or $\wedge q_m$, we get $x \leq y$. 

    We will show that the case $\tau(x) = \tau(y)$ is impossible. Since the sets $\Gamma_{A}, S$ and $I$ are disjoint we may assume that both $x,y$ belong to either $\mathcal{A}$, $\{p_s: s \in S \}$ or $\{q_s : s \in S\}$. Since $x \not= y$, item \ref{faithfulness} in Proposition \ref{coreprop1} shows that at $x,y \in \mathcal{A}$ is impossible. Similarly, Proposition \ref{coreprop1}.\ref{uniquenessincreasing} and Proposition \ref{coreprop1}.\ref{uniquenessdecreasing} show that $x,y \in \{p_s: s \in S \}$ and $x,y \in \{p_s: s \in I \}$ are not possible. This completes the proof that $\mathcal{A}_M$ is an ordered core and completes the proof of existence.

\textbf{Uniqueness:} Let $\mathcal{A}_1$ and $\mathcal{A}_2$ be ordered cores such that $\Gamma_{\mathcal{A}_1} = \overline{\Gamma_{\mathcal{A}}} = \Gamma_{\mathcal{A}_2}$. Let $x \in \mathcal{A}_1 \setminus \mathcal{A}_2$. Since $\Gamma_{\mathcal{A}_1} = \Gamma_{\mathcal{A}_2}$, there must exist some $y \in \mathcal{A}_2$ such that $\tau(y) = \tau(x)$. Item \ref{faithfulness} in Proposition \ref{coreprop1} forces $y$ to not belong to $\mathcal{A}_1$, and the fact that $\mathcal{A} \subseteq \mathcal{A}_1 \cap \mathcal{A}_2$ forces $x$ and $y$ to not be in $\mathcal{A}$ and $\tau(x) \not\in \Gamma_{\mathcal{A}}$. 

If there exists $\{p_n\} \in \mathcal{A}$ such that $\tau(p_n) \uparrow \tau(x)$, then $\{p_n\}$ is an increasing sequence bounded above by $x$. It follows that $\vee p_n \leq x$, however $\tau(\vee p_n) = \sup_n \tau(p_n) = \tau(x)$. By item \ref{faithfulness} in Proposition \ref{coreprop1}, it follows that $x = \vee p_n$. Repeating the same argument with $y$ yields $y = \vee p_n$, hence $x = y$, which is impossible since $x \not\in \mathcal{A}_2$.

If there exists $\{q_n\} \in \mathcal{A}$ such that $\tau(q_n) \downarrow \tau(x)$ an analogous argument as the one used with increasing sequences yields the same contradiction $x = y$. Therefore $\mathcal{A}_2 \subseteq \mathcal{A}_1$. A symmetric argument shows the converse inclusion and proves the equality $\mathcal{A}_1 = \mathcal{A}_2$ completing the proof.

\end{proof}


An ordered core defines a natural subalgebra of $\mathcal{M}$, some of its properties are studied next. We also show that the subalgebra is invariant under the completion of the ordered core done in Theorem \ref{maximal}.

\begin{prop}\label{genabelian}
    Let $(\mathcal{M},\tau)$ be a von Neumann algebra with a semifinite normal faithful trace and $\mathcal{A} \subseteq P(\mathcal{M})$ a $\sigma$-bounded full ordered core. Let $\mathcal{M}_{\mathcal{A}}$ be the von Neumann algebra generated by $\mathcal{A}$. Then $\mathcal{M}_{\mathcal{A}}$ is an abelian subalgebra of $\mathcal{M}$. Moreover, the restriction of $\tau$ to $\mathcal{M}_{\mathcal{A}}$ is a semifinite normal faithful trace.

    Also, if $\mathcal{A}_M$ is the maximal extension from Theorem \ref{maximal}, $\mathcal{M}_{\mathcal{A}} = \mathcal{M}_{\mathcal{A}_M}$.
\end{prop}
\begin{proof}
Consider the set $B = \mathcal{A} \cup \{1_{\mathcal{M}}\}$. Let $V$ be the algebraic $\mathbb{C}$-span of $B$. Since the elements of $B$ are commuting projections, it follows that $V$ is an abelian unital $*$-subalgebra of $\mathcal{M}$. Let $\mathcal{M}_{\mathcal{A}} = V''$. Since $V$ is abelian, then $V \subseteq V'$, two applications of the commutant show that $V'' \subseteq V'''$, so $\mathcal{M}_{\mathcal{A}} \subseteq \mathcal{M}_{\mathcal{A}}'$, hence $\mathcal{M}_\mathcal{A}$ is abelian.  

It is immediate that the restricted trace $\tau$ to $\mathcal{M}_{\mathcal{A}}$ is faithful and normal. To show that it is semifinite, let $x \in (\mathcal{M_{\mathcal{A}}})^+$ and $\{ p_n \}_{n}$ be a sequence in $\mathcal{A}$ such that $p_n \uparrow 1_{\mathcal{M}}$. Since $p_n x p_n \uparrow x$, there exists some $n \in \mathbb{N}$, such that $0 < p_nxp_n$. Since $\mathcal{M}_{\mathcal{A}}$ is abelian, it follows that $p_n x p_n = x p_n$, hence
$$
\tau(p_nxp_n) = \tau(x p_n) \leq \|x\|_{B(H)} \tau(p_n) < \infty. 
$$
This shows that the restricted trace is semifinite.

Finally, if $\mathcal{A}_M$ is the maximal extension of $\mathcal{A}$ given by Theorem \ref{maximal}. Notice that $\mathcal{A} \subseteq P(\mathcal{M}_{\mathcal{A}})$ and $\mathcal{A}_M$ is constructed taking infimum and supremum of sequences in $P(\mathcal{A})$. Since the lattice of projections is complete, it follows that $\mathcal{A}_{M} \subseteq P(\mathcal{M}_{\mathcal{A}})$. Hence $\mathcal{M}_{\mathcal{A}_{M}} \subseteq \mathcal{M}_{\mathcal{A}}$. The converse is immediate, therefore, we have equality. This completes the proof.

\end{proof}

With these tools, we are ready to introduce our notion of a decreasing operator relative to an ordered core. 

\begin{defn}
Given a maximal ordered core $\mathcal{A} \subseteq P(\mathcal{M})$ we say that an element $x \in S(\tau)^+$ is \textit{core decreasing} if there exists a sequence of operators $x_n \in \mathcal{M}^+$ such that $x_n \uparrow x$ and 
$$
x_n = \sum_{k=1}^{m_n} c_{n,k} p_{n,k},
$$
for some $c_{n,k} \geq 0$ and each $p_{n,k} \in \mathcal{A}_M$. We denote by $\mathcal{M}_{\mathcal{A}}^{\downarrow}$ the collection of core decreasing elements in $S(\mathcal{\tau})^+$.
\end{defn}

We finish this section with some examples of ordered cores and their maximal extensions. First, we see an example in abelian von Neumann algebras.

\begin{ex}\label{exCommutative}
    Let $(X,\Sigma,\theta)$ be a $\sigma$-finite measure space and consider $\mathcal{M}$ to be $L^{\infty}(X,\Sigma,\theta)$ acting on the complex Hilbert space $L^{2}(X,\Sigma,\theta)$ by multiplication, equipped with the trace $\tau(f) = \int_{X} f\, d\theta$. Let $S \subseteq \Sigma$ be totally ordered by inclusion $\theta$-a.e. satisfying $\theta(E) < \infty$, for all $E \in S$. The collection $\mathcal{A} = \{0\} \cup \{ p_E\}_{E \in S}$ where $p_E = \chi_{E}$ is a $\sigma$-bounded full ordered core.

    The subalgebra $\mathcal{M}_{\mathcal{A}}$ is the space $L^{\infty}(X,\sigma(S),\theta)$, here $\sigma(S)$ is the $\sigma$-algebra generated by $S$, this subalgebra can be properly contained in $\mathcal{M}$ depending on the space $X$ and the collection $S$. For instance, if $X = \mathbb{R}^n$, $\Sigma$ the Borel $\sigma$-algebra, and $\theta$ the Lebesgue measure, $S_1 = \{ B_r : r \in (0,\infty) \cap \mathbb{Q} \}$, and $S_2 = \{ B_r : r \in (0,\infty)\}$, here $B_r = \{ s \in \mathbb{R}^n: \|s\|_{2} \leq r \}$. Let $\mathcal{A}_1$ and $\mathcal{A}_2$ be the ordered cores induced by $S_1$ and $S_2$.

    It is easy to see that $\Gamma_{\mathcal{A}_1} = [0,\infty) \cap \mathbb{Q}$, and $\Gamma_{\mathcal{A}_2} = [0,\infty)$, it follows that the maximal core extending $\mathcal{A}_1$ is $\mathcal{A}_2$, and that $\mathcal{M}_{\mathcal{A}_1}$ consists of Borel measurable radial functions, which is a strict subalgebra of $\mathcal{M}$ unless $n = 1$. The collection $\mathcal{M}^{\downarrow}_{\mathcal{A}}$ corresponds to radially decreasing functions.

    See \cite[Section~5]{coredecreasing} for more examples using different measure spaces.
\end{ex}

The following example shows an ordered core based on a countable Hilbert basis.
\begin{ex}\label{exHilbert}
    Let $H$ be a separable Hilbert space, with a fixed orthonormal basis $\{ e_n \}_{n}$, $\mathcal{M} = B(H)$, and $\tau$ the standard trace. Define the projections 
    $$
p_n(\xi) = \sum_{k=1}^n \langle \xi , e_k \rangle e_k, \quad \forall n \geq 1. 
    $$
    Set $\mathcal{A} = \{0\} \cup \{p_n\}_{n \geq 1}$. It is clear that $\mathcal{A}$ is a $\sigma$-bounded full ordered core, and that $\tau(p_n) = n$, thus $\Gamma_{\mathcal{A}} = \mathbb{N}$. Therefore, $\mathcal{A}$ is maximal. The subalgebra $\mathcal{M}_{\mathcal{A}}$ is the abelian von Neumann algebra generated by the projections $\xi \mapsto \langle \xi, e_n \rangle e_n$.

    The operators in $\mathcal{M}^{\downarrow}_{\mathcal{A}}$ correspond to increasing limits of operators of the form
    $$
\xi \mapsto \sum_{k=1}^n c_k \langle \xi, e_k \rangle e_k,
    $$
    where $\{c_k\}$ is a decreasing sequence of positive real numbers. This collection includes operators of the form 
    $$
\xi \mapsto \sum_{k=1}^\infty c_k \langle \xi, e_k \rangle e_k,
    $$
    for decreasing sequences $\{c_k\} \in \ell^{\infty}(\mathbb{N})$.
\end{ex}

We finish this section with another example of noncommutative von Neumann algebras.

\begin{ex}\label{exHyperfinite}
    Let $\mathcal{M}$ be the hyperfinite $II_1$ factor, constructed via the sequence of inclusions of matrix algebras $M_{n}(\mathbb{C}) \xhookrightarrow{} M_{2n}(\mathbb{C})$ defined by $x \mapsto \begin{pmatrix}
x & 0 \\
0 & x 
\end{pmatrix}$ extending the trace satisfying $\tau(x) = \frac{1}{2^n} \tau_n$ for $x \in M_{2^n}(\mathbb{C})$. Here, $\tau_n$ is the standard trace in $M_{2^n}(\mathbb{C})$. Let $p_1 = \begin{pmatrix}
1 & 0 \\
0 & 0 
\end{pmatrix}$ and for $n > 2$, define the sequence of projections $p_n \in M_{2^n}(\mathbb{C})$ by the block diagonal matrix
$$
p_n = \begin{pmatrix}
I_2 & 0 & \dots & 0 & 0 \\
0 & I_2 & \dots & 0 & 0 \\
\vdots & \vdots & \ddots & \vdots & \vdots \\
0 & 0 & \dots & I_2 & 0 \\
0 & 0 & \dots & 0 & p_1 \\
\end{pmatrix}, \quad n \geq 2.
$$
Since $p_n \to 1_{\mathcal{M}}$ in strong operator topology, we conclude that $\mathcal{A}$ is a $\sigma$-bounded full ordered core. A direct computation shows that $\tau(p_n) = 1 - 2^{-n}$, therefore $\Gamma_{\mathcal{A}} = \{0\} \cup \{ \sum_{k=1}^n 2^{-k} \}_{n \geq 1}$, so $\mathcal{A}_M = \mathcal{A} \cup \{1_R\}$.
\end{ex}

\section{Morphisms of ordered cores}

We define the notion of a mapping between ordered cores. Theorem \ref{functor} will be one of our main tools for the remaining investigation, which shows that any mapping between ordered cores, induces a linear map between the involved von Neumann algebras.

\begin{defn}
    Let $(\mathcal{M},\tau_1)$ and $(\mathcal{N},\tau_2)$ be von Neumann algebras with semifinite normal faithful traces, let $\mathcal{A} \subseteq P(\mathcal{M})$ be an ordered core. We say that a mapping $r:\mathcal{A} \to P(\mathcal{N})$ is a \textit{core morphism} provided:
    \begin{enumerate}[label=(\roman*)]
        \item We have $r(0) = 0$.
        \item If $p, q \in \mathcal{A}$ and $p \leq q$, then $r(p) \leq r(q)$.
        \item There exists a finite constant $c > 0$ such that if $p, q \in \mathcal{A}$ and $p \leq q$, then $\tau_2\big( r(q) - r(p) \big) \leq c \, \tau_{1}(q-p)$.
        \end{enumerate}
\end{defn}

Our goal is to show that any morphism between ordered cores gives rise to a linear map between their von Neumann algebras. This will be done by inducing a measure on an isomorphic abelian subalgebra of $\mathcal{M}$. We first need two lemmas for that task. Lemma \ref{coremedida} extends \cite[Lemma~4.3]{sinnamon22} to our definition of ordered cores, showing that any ordered core in an abelian von Neumann algebra induces a natural semiring of measurable sets. Lemma \ref{inducedpremeasure} uses a core morphism to pull back integrals to induce a premeasure on that semiring; this is an extension of \cite[Lemma~4.4]{sinnamon22}.

\begin{lem}\label{coremedida}
    Let $\mathcal{A} = \{p_i\}_{i \in I}$ be an ordered core in the abelian von Neumann algebra $L^{\infty}(X,\Sigma,\theta)$ for a $\sigma$-finite measure $\theta$. Let $\Sigma_0 = \{ E \in \Sigma: \theta(E) = 0 \}$, 
\begin{align*}
S = \{ (B \setminus A) \cup C:& \text{ for some } p,q \in \mathcal{A} \text{ such that } p \leq q, p \not= q   \\
&\chi_{B} = q, \chi_{A} = p, \theta \text{ a.e.,} \theta(C) = 0 \text{ and } (B \setminus A) \cap C = \emptyset  \}.
\end{align*}
Set $S^+ = S \cup \Sigma_0$, then:

\begin{enumerate}[label=(\roman*)]
\item \label{measurecoredecomp} If $E \in S^+$, then $\theta(E) > 0$ if and only if $E \in S$.
\item \label{measurecoretotalorder} If $\{ (B_j \setminus A_j) \cup C_j \}_{j \in \mathbb{N}}$ is a sequence in $S$, then there exist sequences $\{A'_{j}\}$, $\{B'_j\}$, $\{C'_j\}$ in $\Sigma$ such that $\chi_{A_j} = \chi_{A'_j}$, $\chi_{B_j} = \chi_{B'_j}$ $\theta$-a.e., $\theta(C'_j) = 0$, $(B_j \setminus A_j) \cup C_j = (B'_j \setminus A'_j) \cup C'_j$, and $(B'_j \setminus A'_j) \cap C'_j = \emptyset$ for each $j \in \mathbb{N}$. Moreover, the collection $\{ A'_j,B'_j\}_{j \in \mathbb{N}}$ is totally ordered by inclusion. And $\theta(E'_j \setminus E'_{k}) = 0$ if and only if $E'_j \subseteq E'_k$ for all $E'_j, E'_k \in \{A'_j\} \cup \{B'_j\}$.
\item \label{measurecorecomplete} Let $E, D \in \Sigma$ such that $\theta(D) = 0$. Then, $E \in S^+$ if and only if $(E \setminus D) \in S^+$.
\item \label{measurewelldefined} If $\{(B_j \setminus A_j) \cup C_j, (B'_j \setminus A'_j) \cup C'_j\} \subseteq S$ and $(B_j \setminus A_j) \cup C_j = (B'_j \setminus A'_j) \cup C'_j$, then $\chi_{B'_j} = \chi_{B_j}$ and $\chi_{A'_j} = \chi_{A_j}$.  
\item \label{measurecoreint} If $E_1$ and $E_2$ belong to $S^+$, then $E_1 \cap E_2 \in S^+$. 
\item \label{measurecoredif} If $E_1$ and $E_2$ belong to $S^+$, then $E_1 \setminus E_2 = \cup_{k=1}^n F_k$, for a disjoint collection $\{F_k\} \in S^+$. 
\end{enumerate}
In particular, $S^+$ is a semiring of sets.
\end{lem}
\begin{proof}
We show \ref{measurecoredecomp} first. Let $E \in S^+$, by construction of $S^+$, it is clear that $\theta(E) > 0$ implies that $E \in S$. Conversely, let $E = \big((B \setminus A) \cup C\big) \in S$ with $p = \chi_{A}$, $q = \chi_{B}$, and suppose that $\theta(E) = 0$ looking for a contradiction. Since $\theta(A) < \infty$ we have 
$$
0 = \theta\big( (B \setminus A) \cup C\big) \geq \theta\big(B \setminus A \big) = \theta(B) - \theta(A).
$$
Since $p \leq q$, we have that $\theta(B) \geq \theta(A)$, so we have $\theta(B) = \theta(A)$. The contradiction $p = q$ follows.

To prove \ref{measurecoretotalorder}, let $\{E_{j}\}_{j \in \mathbb{N}}$ be a collection in $\Sigma$ containing the sequences $\{A_j\},\{B_j\}$ without repetition. For each $(i,j) \in \mathbb{N}^2$ define
$$
G_{jk} = \begin{cases}
			 E_{j} \setminus E_{k}, & \text{if } \theta(E_j \setminus E_k) = 0, \\
            \emptyset, & \text{otherwise.}
		 \end{cases}
$$
Let $G = \cup_{i,j \in \mathbb{N}} G_{ij}$ and define $E'_j = E_{j} \cup G$, $A'_j = E'_{k_j}$ where $k_j$ is the unique index such that $E_{k_j} = A_j$. Similarly, $B'_j = E'_{k_j}$ for the unique $k_j$ such that $B_j = E_{k_j}$ and set
$$
C'_{j} = \Big(\big(G \cap (B_{j} \setminus A_{j})\big) \cup C_j \Big) \setminus (B'_j \setminus A'_{j}).
$$
Notice that $G$ is a countable union of sets of measure zero so $\theta(G) = 0$, $\chi_{E_j} = \chi_{E'_j}$, and $\theta(C'_j) = 0$. If $\chi_{E_j} \leq \chi_{E_k}$ $\theta$-a.e., then $\theta(E_{j} \setminus E_{k}) = 0$, therefore $(E_{j} \setminus E_{k}) = G_{jk} \subseteq G$. Thus,
$$
E'_{j} \setminus E'_k = (E_j \cup G) \setminus (E_k \cup G) = (E_{j} \setminus E_k) \setminus G = \emptyset.
$$
We conclude that $E'_j \subseteq E'_k$. Since for each $i,j$ we have that $\chi_{E_j} \leq \chi_{E_k}$ or $\chi_{E_k} \leq \chi_{E_j}$ $\theta$-.a.e, it follows that $\{E'_j\}_{j \in \mathbb{N}}$ is totally ordered by inclusion.

By \ref{measurecoredecomp}, for each $j \in \mathbb{N}$, we may assume that $\theta(B_j \setminus A_j) > 0$, it follows that $\theta(A_j \setminus B_j) = 0$, therefore
\begin{align*}
(B'_j \setminus A'_j) \cup C'_j &=  \big( (B_j \setminus A_j) \setminus G \big) \cup C'_j \\
&= \big( (B_j \setminus A_j) \setminus G \big) \bigcup \Big(\big(G \cap (B_{j} \setminus A_{j})\big) \cup C_j \Big) \\
&= (B_j \setminus A_j) \cup C_j.  
\end{align*}

For the final property. Since $\{E'_j\}$ is totally ordered by inclusion, it follows that $E'_j \subset E'_k$ implies $\theta(E'_j \setminus E'_k) = 0$. Conversely, suppose that $\theta(E'_j \setminus E'_k) = 0$ for some indices $j,k$. Then,
$$
0 = \theta(E'_j \setminus E'_k) = \theta\big( (E_j \setminus E_k) \setminus G \big).
$$
Since $\theta(G) = 0$, it follows that $\theta(E_j \setminus E_k) = 0$, hence $E_j \setminus E_k \subseteq G_{ij} \subseteq G$. Hence, $E'_j \setminus E'_k = \emptyset$ and completes the proof of \ref{measurecoretotalorder}.

To prove \ref{measurecorecomplete}, let $E, D \in \Sigma$ such that $\theta(D) = 0$. If $E \in \Sigma_0$, the statement is obvious. Therefore we may assume that $E \in S$, so $E = (B \setminus A) \cup C$, hence
\begin{align*}
    E \setminus D &= \big( (B \setminus A) \cup C \big) \setminus D = \big( (B \setminus A) \cup C \big) \cap D^c \\
    &= \big( (B \setminus A) \cap D^c \big) \cup (C \cap D^c) = \big( B \cap A^c \cap D^c \big) \cup (C \cap D^c) \\
    &= \big( B \cap (A \cup D)^c \big) \cup (C \cap D^c) = \big( B \setminus (A \cup D) \big) \cup (C \setminus D).
\end{align*}
Let $A' = A \cup D$, note that $\chi_{A'} = \chi_{A}$ since $\theta(D) = 0$ and also that $\theta(C \setminus D) = 0$, therefore $E \setminus D = (B \setminus A') \cup (C \setminus D)$ which belongs to $S$ since the union is disjoint.  

Conversely, if $E \setminus D \in S^+$, then $E \setminus D = (B \setminus A) \cup C$. Hence
\begin{align*}    
E &= (E \setminus D) \cup (E \cap D) = (B \setminus A) \cup C \cup (E \cap D)\\ 
&= (B \setminus A) \cup \Big( C \cup (E \cap D) \Big)
= (B \setminus A) \cup C'.
\end{align*}
Here $C'$ is the set of zero measure $\Big( C \cup (E \cap D) \Big) \setminus (B \setminus A)$. Thus, $E \in S^+$ and proves \ref{measurecorecomplete}.

Now, we prove \ref{measurewelldefined}. First, notice that by \ref{measurecoretotalorder} we may assume that $\{A'_j,B'_j,A_j,B_j\}$ is totally ordered by inclusion. Suppose that $\chi_{B_j} \leq \chi_{B'_j}$ $\theta$-a.e., notice that $\theta(B'_j \setminus A'_j) = \theta(B_j \setminus A_j)$. If $\theta(B_j) < \theta(B'_j)$ then
$$
\theta(B'_j) - \theta(A'_j) = \theta(B'_j \setminus A'_j) = \theta(B_j \setminus A_j) = \theta(B_j) - \theta(A_j) < \theta(B'_j) - \theta(A_j). 
$$
It follows that $\theta(A'_j) > \theta(A_j)$. So $\theta\big((A'_j \setminus A_j) \cap B_j \big) = \min\{\theta(B_j),\theta(A'_j)\} - \theta(A_j) > 0$, however
$$
(A'_j \setminus A_j) \cap B_j \subseteq (B_j \setminus A_j) \cup C_j = (B'_j \setminus A'_j) \cup C'_j.
$$
It follows that $(A'_j \setminus A_j) \cap B_j$ is contained in the set of measure zero $C'_j$, which is impossible. Therefore $\theta(B_j) = \theta(B'_j)$ and using the fact that $\theta(B'_j \setminus A'_j) = \theta(B_j \setminus A_j)$ we conclude that $\theta(A_j) = \theta(A'_j)$. Hence $\chi_{B_j} = \chi_{B'_j}$ and $\chi_{A_j} = \chi_{A'_j}$, as desired.

To prove \ref{measurecoreint} let $E_1, E_2 \in S^+$. If either has measure zero, then $E_1 \cap E_2 \in \Sigma_0 \subseteq S^+$. So we may assume that $E_1, E_2 \in S$. Let $E_1 = (B_1 \setminus A_1) \cup C_1$ and $E_2 = (B_2 \setminus A_2) \cup C_2$. By \ref{measurecoretotalorder} we may assume that $\{A_1,A_2,B_1,B_2\}$ are totally ordered by inclusion. Note that we must have that $A_1 \subset B_1$ and $A_2 \subset B_2$. After a relabelling, we may assume that $B_1 \subseteq B_2$, set $E = E_1 \cap E_2$, then
    \begin{align*}
        E &= \big((B_1 \setminus A_1) \cup C_1\big) \cap \big((B_2 \setminus A_2) \cup C_2\big) \\
        &= \big((B_1 \setminus A_1) \cap (B_2 \setminus A_2) \big) \\ &\cup \bigg( \big( (B_1 \setminus A_1) \cap C_{2} \big) \cup \Big(C_1 \cap \big( (B_2 \setminus A_2) \cup C_2 \big) \Big)  \bigg)\\
        &= \big((B_1 \setminus A_1) \cap (B_2 \setminus A_2) \big) \cup D,
    \end{align*}
    here $D$ is the set of measure zero $\big( (B_1 \setminus A_1) \cap C_{2} \big) \cup \Big(C_1 \cap \big( (B_2 \setminus A_2) \cup C_2 \big) \Big)$. 

    If $B_1 \subseteq A_{2}$, then $E = D$, and belongs to $S^+$ since it would have measure zero. In the case $A_2 \subseteq B_1$ we have
    $$
E = \big(B_{1} \setminus (A_{1} \cup A_{2}) \big) \cup D,
    $$
    since $A_{1} \cup A_{2} = A_1$ or $A_{1} \cup A_{2} = A_2$ we conclude that $E \in S^+$, this completes the proof of \ref{measurecoreint}.

    Finally, to prove \ref{measurecoredif}. Let $F = E_1 \setminus E_2$, if $E_1$ has measure zero, then $\theta(F) = 0$, so $F \in S^+$. If $E_2$ has measure zero, then $F \in S^+$, as a consequence of \ref{measurecorecomplete}. So we may assume that $E_1 = (B'_1 \setminus A'_1) \cup C'_1$ and $E_2 = (B'_2 \setminus A'_2) \cup C'_2$. Once more, we may assume that $A'_1 \subset B'_1$, $A'_2 \subset B'_2$ and $B'_1 \subseteq B'_2$.
    
    Notice that 
    \begin{align*}
(B'_1 \setminus A'_1) \setminus (B'_2 \setminus A'_2) &= \big(B'_1 \cap (A'_1)^c \big) \cap \big(B'_2 \cap (A'_2)^c \big)^c \\
&= \big(B'_1 \cap (A'_1)^c \big) \cap \big((B'_2)^c \cup (A'_2) \big) \\
&= \big(B'_1 \cap (A'_1)^c \cap (B'_2)^c \big) \cup \big( B'_1 \cap (A'_1)^c \cap A'_2 \big)\\
&= \big(B'_1 \cap \big( A'_1 \cup B'_2\big)^c \big) \cup \big( (B'_1 \cap A'_2) \cap (A'_1)^c \big)\\
&= \big(B'_1 \setminus \big( A'_1 \cup B'_2\big) \big) \cup \big( (B'_1 \cap A'_2) \setminus A'_1 \big).
    \end{align*}
Since $A'_1 \cup B'_2 = A'_1$ or $A'_1 \cup B'_2 = B'_2$ and $B'_1 \cup A'_2 = B'_1$ or $B'_1 \cup A'_2 = A'_2$, we conclude that $B'_1 \setminus \big( A'_1 \cup B'_2\big) \in S^+$ and $(B'_1 \cap A'_2) \setminus A'_1 \in S^+$. Moreover, we assumed that $A'_2 \subset B'_2$, therefore  $\big(B'_1 \setminus \big( A'_1 \cup B'_2\big) \big) \cup \big( (B'_1 \cap A'_2) \setminus A'_1 \big)$ is a disjoint union in $S^+$.

Then
\begin{align*}
    F &= \big((B'_1 \setminus A'_1) \cup C'_1\big) \setminus \big((B'_2 \setminus A'_2) \cup C'_2\big) \\
    &= \Big((B'_1 \setminus A'_1) \setminus \big((B'_2 \setminus A'_2) \cup C'_2\big) \Big) \cup \Big(C'_1 \setminus \big((B'_2 \setminus A'_2) \cup C'_2\big) \Big)  \\
    &= \Big( \big((B'_1 \setminus A'_1) \setminus (B'_2 \setminus A'_2) \big) \setminus C'_2 \Big) \cup \Big(C'_1 \setminus \big((B'_2 \setminus A'_2) \cup C'_2\big) \Big)\\
    &= \bigg( \Big( \big(B'_1 \setminus \big( A'_1 \cup B'_2\big) \big) \cup \big( (B'_1 \cap A'_2) \setminus A'_1 \big) \Big) \setminus C'_2 \bigg) \\
    &\cup \Big(C'_1 \setminus \big((B'_2 \setminus A'_2) \cup C'_2\big) \Big) \\
    &= \Big( \big(B'_1 \setminus  (A'_1 \cup B'_2) \big) \setminus C'_2 \Big) \cup \Big( \big( (B'_1 \cap A'_2) \setminus A'_1 \big) \setminus C'_2 \Big) \cup H, 
\end{align*}
here $$H = \Big(C'_1 \setminus \big((B'_2 \setminus A'_2) \cup C'_2\big) \Big) \setminus \bigg( \Big( \big(B'_1 \setminus  (A'_1 \cup B'_2) \big) \setminus C'_2 \Big) \cup \Big( \big( (B'_1 \cap A'_2) \setminus A'_1 \big) \setminus C'_2 \Big) \bigg),$$ which is a set of zero measure so it belongs to $S^+$.
Then, $F$ is the disjoint union of the sets $H$, $\big(B'_1 \setminus  (A'_1 \cup B'_2) \big) \setminus C'_2$ and $\big( (B'_1 \cap A'_2) \setminus A'_1 \big) \setminus C'_2$, with the last two sets belonging to $S^+$ by \ref{measurecorecomplete}. This shows that $S^+$ is a semiring of sets and completes the proof.

\end{proof}
As an immediate application of the previous lemma, we get a characterization of equality of positive elements in $\mathcal{M}_{\mathcal{A}}$.
\begin{lem}\label{lemaigualador}
    Let $\mathcal{M}$ be a von Neumann algebra equipped with a semifinite normal faithful trace $\tau$ and a full $\sigma$-bounded ordered core $\mathcal{A}$. Let $\mathcal{M}_{\mathcal{A}}$ be the von Neumann algebra generated by $\mathcal{A}$ and $x,y \in \mathcal{M}^+_{\mathcal{A}}$. Then $x = y$ if and only if
    $$
\tau(xp) = \tau(yp), \quad \forall p \in \mathcal{A}.
    $$
\end{lem}
\begin{proof}
Suppose that $x,y \in \mathcal{M}_{\mathcal{A}}^+$ such that $\tau(xp) = \tau(yp)$. By Lemma \ref{genabelian} and Theorem \ref{abeliano}, there exists a trace preserving $*$-isomorphism $\Phi: \mathcal{M}_{\mathcal{A}} \to L^{\infty}(X,\Sigma,\theta)$. The collection $\Phi(\mathcal{A})$ is an ordered core in $L^{\infty}(X,\Sigma,\theta)$. 
Let $\sigma(\Phi(\mathcal{A}))$ be the $\sigma$-algebra generated by $\Phi(\mathcal{A})$. Since $L^{\infty}(X,\sigma(\Phi(\mathcal{A})),\theta)$ is a von Neumann subalgebra of $(X,\Sigma,\theta)$, then $\Phi^{-1}\Big( L^{\infty}(X,\sigma(\Phi(\mathcal{A})),\theta)\Big)$ is a von Neumann subalgebra of $\mathcal{M}_{\mathcal{A}}$ containing $\mathcal{A}$. Therefore, $L^{\infty}(X,\sigma(\Phi(\mathcal{A})),\theta)$ coincides with $L^{\infty}(X,\Sigma,\theta)$. 

The measures $E \mapsto \int_{E} \Phi(x) \, d\theta$ and $E \mapsto \int_{E} \Phi(y) \, d\theta$ coincide on the semiring $S^+$ defined in Lemma \ref{coremedida}, hence, they coincide on the generated $\sigma$-algebra $\sigma(\Phi(\mathcal{A})) = \Sigma$. By the uniqueness of the Radon-Nikodym derivative, we get that $\Phi(x) = \Phi(y)$. Hence, $x = y$. This completes the proof of the nontrivial implication of the lemma.

\end{proof}

The next lemma shows that given a core morphism, we can induce a measure on a related measurable space.

\begin{lem}\label{inducedpremeasure}
        Let $(\mathcal{M},\tau_1)$ and $(\mathcal{N},\tau_2)$ be von Neumann algebras with semifinite normal faithful traces, let $\mathcal{A} \subseteq P(\mathcal{M})$ be an ordered core and a core morphism $r:\mathcal{A} \to P(\mathcal{N})$ with constant $c$. Let $\Phi: \mathcal{M}_{\mathcal{A}} \to L^{\infty}(X,\Sigma,\theta)$ be the $*$-isomorphism given by Theorem \ref{abeliano}, the sets $\Sigma_0 = \{ E \in \Sigma: \theta(E) = 0\}$,
\begin{align*}
S = \{ (B \setminus A) &\cup C: \text{ for some } p,q \in \mathcal{A} \text{ such that } p \leq q, p \not= q   \\
&\chi_{B} = \Phi(q), \chi_{A} = \Phi(p) \, \theta \text{- a.e.,} \theta(C) = 0 \text{ and } (B \setminus A) \cap C = \emptyset  \},
\end{align*}
and $S^+ = \Sigma_0 \cup S$.
    Then, for a fixed $y \in N$, the formula
    $$
\rho_y\big( (B \setminus A) \cup C \big) = 			 \tau_{2} \Big( y \big(r(q) - r(p)\big) \Big) \quad \text{if } \Phi(q) = \chi_B, \Phi(p) = \chi_{A} \text{ and } p < q,$$
    and $\rho_y(E) = 0$ if $E \in \Sigma_0$, defines a premeasure on the semiring $S^+$. Moreover, there is a unique $\sigma$-finite measure, defined on the $\sigma$-algebra generated by $\Phi(\mathcal{A})$, that extends $\rho_y$ and is absolutely continuous with respect to $\theta$.
\end{lem}
\begin{proof}
Since $\Phi$ is a $*$-isomorphism, every element in $\{ \Phi(p): p \in \mathcal{A} \}$ is a totally ordered set of projections. Since
$$
\int_{X} \Phi(p) \, d\theta = \tau(p) < \infty,
$$
we conclude that the collection $\{ \Phi(p): p \in \mathcal{A} \}$ is an ordered core in the abelian von Neumann algebra $L^{\infty}(X,\Sigma,\theta)$. Hence, Lemma \ref{coremedida} shows that $S^+$ is a semiring of sets.

    First, we show that $\rho_y$ is well defined. Let $E \in S^+$. By Lemma \ref{coremedida} \ref{measurecoredecomp}, either $E$ has measure zero or is of the form $(B \setminus A) \cup C$ for $C \in \Sigma_0$, $\chi_B = \Phi(q)$ and $\chi_A = \Phi(p)$ for some $p,q \in \mathcal{A}$. Notice that $p < q$, otherwise, $E$ must have measure zero.
    
    To check $\rho_y$ is well defined, let $E = (B' \setminus A') \cup C'$ such that $\theta(C') = 0$, $\chi_{B'} = \Phi(q')$ and $\chi_{A'} = \Phi(p')$ for $q',p' \in \mathcal{A}$. By Lemma \ref{coremedida} \ref{measurewelldefined} we have that $\Phi(p) = \Phi(p')$ and $\Phi(q) = \Phi(q')$. It follows that $p = p'$ and $q = q'$ so 
$$
\tau_2\Big( y\big( r(q) - r(p) \big)\Big) = \tau_2\Big( y\big( r(q') - r(p') \big)\Big).
$$
This proves that $\rho_y$ is a well-defined function. 

Since $r$ is order-preserving and $y$ is positive, it follows that $\rho_y$ takes nonnegative values. It is also clear that $\rho_y(\emptyset) = 0$.

We now show that $\rho_y$ is finitely additive. Let $E, \{E_1,\dots,E_n\} \in S^+$ such that $E = \cup_{k=1}^n E_k$ and the union is disjoint. We prove that $\rho_y(E) = \sum_{k=1}^n \rho_y(E_k)$. If $n=1$ there is nothing to prove, as $\rho_y$ is well defined. If $\theta(E) = 0$ then the measure of every set involved is zero and the equality is clear. If $\theta(E_k) = 0$ for some $k \in \{1,\dots,n\}$, then \ref{coremedida} \ref{measurecorecomplete} shows that $E \setminus E_k \in S^+$. Also, the proof of \ref{measurecorecomplete} shows that if $E = (B \setminus A) \cup C$, then $E \setminus E_k = (B \setminus A') \cup C'$, where $\chi_A = \chi_{A'}$ and $\theta(C') = 0$. It follows that $\rho_y(E \setminus E_k) = \rho_y(E)$ and $\rho_y(E_k) = 0$, so we can ignore the set of measure zero. After finitely many steps we may remove all sets of measure zero from the union.

Therefore, we may assume that $\theta(E_k) > 0$ for each $k \in \{1,\dots,n\}$. Let $E = (B \setminus A) \cup C$ and $E_k = (B_k \setminus A_k) \cup C_k$. By Lemma \ref{coremedida} \ref{measurecoretotalorder} we may assume that all the sets of positive measure are totally ordered by inclusion. Let $D = \cup_{k} C_k$, then we have
$$
(B \setminus A) \cup C = D \cup \bigcup_{k=1}^n (B_k \setminus A_k),
$$
and the union is disjoint. Since the union is disjoint, after a relabelling, we may assume that $A_1 \subset B_1 \subseteq A_2 \subset B_2 \subseteq \dots \subseteq A_n \subset B_n$. We now show that $B = B_n$ and that $A = A_1$. 

Suppose that $B \subseteq B_n$, then $B_n \setminus (B \cup A_n) \subseteq C$, it follows that $\theta(B_n \setminus (B \cup A_n)) = 0$. By virtue of Lemma \ref{coremedida}.\ref{measurecoretotalorder}, we have that $B_n \subseteq (B \cup A_n)$. Since $\theta(B_n \setminus A_n) > 0$, we are forced to have that $B_n = B$. If $B_n \subseteq B$, then $(B \setminus B_n) \subseteq D$, by the same argument as before, we conclude that $B = B_n$.

Suppose that $A \subseteq A_1$, then $A_1 \setminus A \subseteq C$, it follows that $\theta(A_1 \setminus A) = 0$, and by the same argument as before, we are forced to have that $A_1 = A$. If $A_1 \subseteq A$, then $(B_1 \cap A) \setminus A_1 \subseteq D$, and again, we conclude that $A = A_1$.

Since $\theta(C) = 0 = \theta(D)$ we have
\begin{align*}
\theta(B_n) - \theta(A_1) &= \theta(B) - \theta(A) = \theta\big((B \setminus A) \cup C\big) = \theta\Big( D \cup \bigcup_{k=1}^n  (B_k \setminus A_k) \Big)\\
&= \sum_{k=1}^n \theta(B_k \setminus A_k) = \sum_{k=1}^n \big(\theta(B_k) - \theta(A_k)\big) \\
& = \theta(B_n) - \theta(A_1) - \sum_{k=1}^{n-1} \Big(\theta(A_{k+1}) - \theta(B_k) \Big).
\end{align*}
It follows that $\theta(A_{k+1} \setminus \theta(B_k)) = 0$ for each $k \in \{1,\dots,n-1\}$ and Lemma \ref{coremedida} \ref{measurecoretotalorder} shows that $B_k = A_{k+1}$ for each $k \in \{1,\dots,n-1\}$. We have that $\Phi(q_n) = \chi_{B_n} = \chi_{B} = \Phi(q)$, $\Phi(p_1) = \chi_{A_1} = \chi_{A} = \Phi(p)$, and  $\Phi(q_k) = \chi_{B_k} = \chi_{A_{k+1}} = \Phi(p_{k+1})$ for each $k \in \{1,\dots,n-1\}$. It follows that $p_{k+1} = q_k$ for all $k \in \{1,\dots,n-1\}$. Hence,
\begin{align*}
    \sum_{k=1}^n \rho_{y}(E_k) &= \sum_{k=1}^n \tau_2\Big(y \big(r(q_k) - r(p_k) \big)\Big) = \sum_{k=1}^n \Big( \tau_2\big(y r(q_k) \big) - \tau_2\big(y r(p_k) \big)  \Big) \\
    &= \tau_2\big(y r(q_n)\big) - \tau_2\big(y r(p_1)\big) - \sum_{k=1}^{n-1} \Big( \tau_2\big(y r(p_{k+1}) \big) - \tau_2\big(y r(q_k) \big)  \Big)\\
    &= \tau_2\big(y r(q_n)\big) - \tau_2\big(y r(p_1)\big) = \tau_2\Big(y \big(r(q) - r(p) \big)\Big) = \rho_y(E).
\end{align*}

Notice that every term is finite, since
\begin{align*}
\tau_2\Big(y r(p)\Big) &= \tau_2\Big(y \big(r(p) - r(0)\big) \Big) \\
&\leq  \|y\|_{B(H_2)} \,  \tau_2\big(r(p) - r(0)\big) \leq c \|y\|_{B(H_2)} \,  \tau_1(p) < \infty.    
\end{align*}
Thus, we have proven that $\rho_y$ is finitely additive.
 
We now show that $\rho_y$ is finitely monotone. Let $E, \{E_k\}_{k=1}^n \in S^+$, such that $E \subseteq \cup_k E_k$. We will show that $\rho_y(E) \leq \sum_{k=1}^n \rho_y(E_k)$. By the same argument used before, we may assume that $\theta(E) > 0$ and that $\theta(E_k) > 0$. Lemma \ref{coremedida} \ref{measurecoretotalorder} we may pick a totally ordered collection of sets $\{A,B,\{A_k\},\{B_k\}\}$ such that $\Phi(p) = \chi_{A}$, $\Phi(q) = \chi_B$, $\Phi(p_k) = \chi_{A_k}$, $\Phi(q_k) = \chi_{B_k}$ for some $\{p,q,\{p_k\},\{q_k\}\} \in \mathcal{A}$, $E = (B \setminus A) \cup C$ and $E_k = (B_k \setminus A_k) \cup C_k$ with $\theta(C) = 0 = \theta(C_k)$. After a relabelling we may assume that $A_1 \subseteq A_2 \subseteq \dots \subseteq A_n$. Let $D = \cup_{k=1}^n C_k$, we have
$$
(B \setminus A) \cup C \subseteq D \cup \bigcup_{k=1}^n (B_k \setminus A_k).
$$
Since $\rho_y\big( B \setminus A\big) = \rho_y(E)$, without loss of generality, we may suppose that $C = \emptyset$. We will show that there exists a collection of sets $\{B_{jk},A_{jk}\}_{j=0,k=0}^n \subseteq \{A_j,B_j\}_{j=1}^n$ such that for each $j \in \{0,\dots,n\}$ we have $A_{jk} \subseteq B_{jk}$ for all $k$, $A_{jk_1} \subseteq A_{jk_2}$ for $k_1 < k_2$, $\bigcup_{k=1}^n (B_{jk} \setminus A_{jk}) = \bigcup_{k=1}^n (B_k \setminus A_k)$, also $\rho_y\big( B_{jk} \setminus A_{jk} \big) \leq \rho_y\big( B_{(j-1)k} \setminus A_{(j-1)k} \big)$ for all $k \in \{1,\dots,n\}$, and $\big( B_{jk} \setminus A_{jk} \big) \cap \big( B_{jk'} \setminus A_{jk'} \big) = \emptyset$ for $k \leq j$ and $k' > k$. We proceed by induction over $j$.

For the base case, set $B_{00} = B_n$, $A_{00} = A_1$, $B_{0k} = B_k$ and $A_{0k} = A_k$ for $k \in \{1,\dots,n\}$. All the conditions follow immediately. So we prove the inductive step. Set $B_{jk} = B_{(j-1)k}$ and $A_{jk} = A_{(j-1)k}$ for each $k \in \{0,\dots,j\}$. For $k' > j$ define
$$
B_{jk'} = \begin{cases}
            A_{(j-1)k'}, & \text{if } B_{(j-1)k'} \subseteq B_{(j-1)j},\\
            B_{(j-1)k'}, & \text{otherwise.}
		 \end{cases}
$$
and
$$
A_{jk'} = \begin{cases}
            B_{(j-1)j}, & \text{if } A_{(j-1)k'} \subseteq B_{(j-1)j} \subseteq B_{(j-1)k'},\\
            A_{(j-1)k'}, & \text{otherwise.}
		 \end{cases}
$$

Notice that for $k' > j$ we have $(B_{jk'} \setminus A_{jk'}) \subseteq (B_{(j-1)k'} \setminus A_{(j-1)k'})$, therefore we have
\begin{align*}
\rho_y\Big( B_{jk'} \setminus A_{jk'} \Big) &= \tau_2\Big(y \big( r(q_{jk'}) - r(p_{jk'}) \big)\Big) \\
&\leq \tau_2\Big(y \big( r(q_{(j-1)k'}) - r(p_{(j-1)k'}) \big)\Big) \\
&= \rho_y\Big( B_{(j-1)k'} \setminus A_{(j-1)k'} \Big).
\end{align*}
And if $k \leq j$ we have 
$$
(B_{jk} \setminus A_{jk}) \cap (B_{jk'} \setminus A_{jk'}) \subseteq (B_{(j-1)k} \setminus A_{(j-1)k}) \cap (B_{(j-1)k'} \setminus A_{(j-1)k'}) = \emptyset. 
$$
It also follows from the construction that
$$
(B_{jj} \setminus A_{jj}) \cup (B_{jk'} \setminus A_{jk'}) = (B_{(j-1)j} \setminus A_{(j-1)j}) \cup (B_{(j-1)k'} \setminus A_{(j-1)k'}),   
$$
hence
\begin{align*}
\bigcup_{k=1}^n (B_{jk} \setminus A_{jk}) &= \bigcup_{k=1}^{j-1} (B_{jk} \setminus A_{jk}) \cup \bigcup_{k'=j+1}^n \Big( (B_{jj} \setminus A_{jj}) \cup (B_{jk'} \setminus A_{jk'}) \Big) \\
&= \bigcup_{k=1}^{j-1} (B_{(j-1)k} \setminus A_{(j-1)k}) \\
&\cup \bigcup_{k'=j+1}^n \Big( (B_{(j-1)j} \setminus A_{(j-1)j}) \cup (B_{(j-1)k'} \setminus A_{(j-1)k'}) \Big) \\
&= \bigcup_{k=1}^{n} (B_{(j-1)k} \setminus A_{(j-1)k}) = \bigcup_{k=1}^{n} (B_{k} \setminus A_{k}). 
\end{align*}

By the induction hypothesis $A_{jk_1} \subseteq A_{j k_2} \subseteq A_{jj}$ for all $k_1 < k_2 \leq j$. For $k' \in \{j+1,\dots,n\}$ we may relabel the sets to make sure that $\{A_{jk'}\}_{k'=j+1}^{n}$ is increasing by inclusion. This finishes the induction.

We get
$$
(B \setminus A) \subseteq D \cup \bigcup_{k=1}^n (B_{nk} \setminus A_{nk}) = D \cup \bigcup_{k=1}^n (B_k \setminus A_k).
$$
Once more, total order and the fact that $\theta(D) = 0$ forces that $B \subseteq B_{nn}$ and $A_{n1} \subseteq A$, so 
$$
(B \setminus A) \subseteq (B_{nn} \setminus A_{n1}) \cup D =  D \cup \bigcup_{k=1}^n (B_{nk} \setminus A_{nk}). 
$$
Since $\rho_y$ is finitely additive, and the collection is disjoint, we have that
$$
\rho_y\big( (B_{nn} \setminus A_{n1}) \cup D \big) = \sum_{k=1}^{n} \rho_y\big( B_{nk} \setminus A_{nk} \big), 
$$
hence
\begin{align*}
    \rho_y(E) &= \rho_y(B \setminus A) = \tau_2\Big( y \big( r(q) - r(p) \big)\Big) \leq \tau_2\Big( y \big( r(q_{nn}) - r(p_{n1}) \big)\Big) \\ &= \rho_y\big( (B_{nn} \setminus A_{n1}) \cup D \big) = \sum_{k=1}^{n} \rho_y\big( B_{nk} \setminus A_{nk} \big)\\
&\leq \sum_{k=1}^{n} \rho_y\big( B_{k} \setminus A_{k} \big) = \sum_{k=1}^n \rho_y(E_k).
\end{align*}
This proves that $\rho_y$ is finitely monotone.

We now show that $\rho_y$ is countably monotone. Suppose that $E$, $\{E_k\}_{k=1}^\infty$  belong to $S^+$ such that $E \subseteq \bigcup_{k=1}^\infty E_k$. Our goal is to show that $\rho_y(E) \leq \sum_{k=1}^\infty \rho_y(E_k)$. 

For each $n$, after finitely many applications of Lemma \ref{coremedida} \ref{measurecoredif} we may find a finite disjoint collection $\{E'_j\}_{j = 1}^m \in S^+$ such that
$$
E \setminus \Big(\bigcup_{k=1}^n E_{k}\Big) = \bigcup_{j=1}^m E'_j.  
$$
Using Lemma \ref{coremedida} \ref{measurecoretotalorder}, we may write $E'_j = (B'_j \setminus A'_j) \cup C'_j$ with $A'_j \subseteq B'_j$, $\theta(C'_j) = 0$, $\{ A'_j\} \cup \{B'_j\}$ totally ordered by inclusion, $\{p'_j\} \cup \{q'_j\} \subseteq \mathcal{A}$, $\Phi(p'_j) = \chi_{A'_{j}}$, and $\Phi(q'_j) = \chi_{B'_{j}}$. 

Hence
\begin{align*}
\rho_y(E'_j) &= \tau_2\Big(y \big( r(q'_j) - r(p'_j) \big) \Big) \leq \|y\|_{B(H_2)} \, \tau_2\Big(\big( r(q'_j) - r(p'_j) \big) \Big) \\
&\leq c \|y\|_{B(H_2)} \, \tau_1\Big(\big( q'_j - p'_j \big) \Big) = c \|y\|_{B(H_2)} \, \theta(B'_j \setminus A'_j) \\
&= c \|y\|_{B(H_2)} \, \theta(E'_j).  
\end{align*}
Addition over $j$ yields
\begin{align*}
\sum_{j=1}^m \rho_y(E'_j) &\leq c \|y\|_{B(H_2)} \, \sum_{j=1}^m \theta(E'_j) = c \|y\|_{B(H_2)} \, \theta\Big(\cup_{j=1}^m E'_j\Big) \\
&= c \|y\|_{B(H_2)} \, \theta\Big(E \setminus \cup_{k=1}^n E_k\Big).  
\end{align*}
Since $E \subseteq \bigcup_{k=1}^{n} E_k \cup \bigcup_{j=1}^m E'_j$, finite monotonicity yields
\begin{align*}
    \rho_y(E) &\leq \sum_{k=1}^n \rho_y(E_k) + \sum_{j=1}^m \rho_y(E'_j) \leq \sum_{k=1}^n \rho_y(E_k) + c \|y\|_{B(H_2)} \, \theta\Big(E \setminus \cup_{k=1}^n E_k\Big).
\end{align*}
Letting $n \to \infty$ and noting that $\theta(E \setminus \cup_{k=1}^n E_k) \to 0$ we get $\rho_y(E) \leq \sum_{k=1}^\infty \rho_y{E_k}$. This shows that $\rho_y$ is countably monotone and proves that $\rho_y$ is a premeasure on the semiring $S^+$. 

We now show that $\rho_y$ is $\sigma$-finite. Notice that for any $E \in S^+$ either $\rho_y(E) = 0$ or $\rho_y(E) = \tau_2\Big(y \big(r(q) - r(p) \big) \Big)$ for some $p,q \in \mathcal{A}$ satisfying $p \leq q$. Let $\{p_n\}$ be a sequence in $\mathcal{A}$ such that $p_n \uparrow 1_{\mathcal{M}}$. Then $\Phi(p_n) \uparrow 1$. Let $E_n \in \Sigma$ satisfy $\Phi(p_n) = \chi_{E_n}$. Then $E_n \uparrow X$ and
$$
\rho_y(E_n) = \tau_2(y \, r(p_n)) \leq \|y\|_{B(H_2)} \tau_2(r(p_n)) \leq c \|y\|_{B(H_2)} \tau_1(p_n) < \infty.
$$
Therefore, $\rho_y$ is $\sigma$-finite.


Let $(\rho_y)^*$ be the outer measure induced by $\rho_y$. By the Caratheodory-Hahn theorem (\cite[Page~361]{royden}), there exists a unique $\sigma$-finite measure $\overline{\rho_y}$ defined over $(\rho_y)^*$-measurable sets extending $\rho_y$. In particular, it extends $\rho_y$ to the $\sigma$-ring generated by $S^+$.

To show that $\rho_y$ is absolutely continuous with respect to $\theta$. Let $E$ be a $\sigma(S^+)$-measurable set such that $\theta(E) = 0$. Fix $\epsilon > 0$, then there exists a sequence $\{ E_n \} \in S^+$, such that $E \subseteq \cup_n E_n$ and $\sum_{n} \theta(E_n) < \frac{\epsilon}{c \|y\|_{B(H_2)}}$. By Lemma \ref{coremedida}.\ref{measurecoretotalorder} we may assume that $E_n = (B_n \setminus A_n) \cup C_n$ such that $\theta(C_n) = 0$, $\chi_{B_n} = \Phi(p_n)$ and $\chi_{A_n} = \Phi(p_n)$. Then
\begin{align*}
\rho_y(E) &\leq \sum_{n}\rho_y(E_n) = \sum_{n} \tau_2\Big( y \big( r(p_n) - r(q_n) \big)\Big) \\
&\leq \sum_{n} c\|y\|_{B(H_2)} \tau_1\Big( p_n - q_n\Big) \\
&= c\|y\|_{B(H_2)} \sum_{n} \theta\big( B_n \setminus A_n \big) = c\|y\|_{B(H_2)} \sum_{n} \theta\big(E_n\big) < \epsilon. 
\end{align*}
Letting $\epsilon \to 0$ shows that $\rho_y(E) = 0$, which shows that $\rho_y$ is absolutely continuous with respect to $\theta$ and completes the proof.
     
\end{proof}

We now show the main result of this section, that a core morphism induces a linear map on the von Neumann algebras involved.

\begin{thm}\label{functor}
Let $(\mathcal{M},\tau_1)$ and $(\mathcal{N},\tau_2)$ be von Neumann algebras with semifinite normal faithful traces, let $\mathcal{A} \subseteq P(\mathcal{M})$ be a $\sigma$-bounded full ordered core and a core morphism $r:\mathcal{A} \to P(\mathcal{N})$ with constant $c$. Then, there exists a linear $*$-preserving map $\vartheta: \mathcal{N} \to \mathcal{M}$ such that:
    \begin{enumerate}[label=(\roman*)]
        \item \label{functorpositive} If $y \in \mathcal{N}^+$, then $\vartheta(y) \in \mathcal{M}^+$.
        \item \label{functorlinf} For each $y \in \mathcal{N}^+$, $\|\vartheta(y)\|_{B(H_1)} \leq c \|y\|_{B(H_2)}$.
        \item \label{functorincreasing} If $\{y_n\}_n$ and $y$ belong to $\mathcal{N}^+$ such that $y_n \uparrow y$, then $\vartheta(y_n) \uparrow \vartheta(y)$.

        \item \label{functorlinear} For each $y \in \mathcal{N}$ and all $p,q \in \mathcal{A}$ such that $p \leq q$, the equality
        \begin{equation}\label{taufunctor}
            \tau_1\big( (q-p) \vartheta(y)\big) = \tau_2\Big( \big( r(q) - r(p) \big) y \Big).
        \end{equation}
        holds.
        \item \label{functorcoreproj} If $c = 1$, $p \in \mathcal{A}$, and $\tau_2(r(p)) = \tau_1(p)$, then $\vartheta(r(p)) = p$.
        \item \label{functormult} If $c = 1$, $\tau_2(r(p)) = \tau_1(p)$ for all $p \in \mathcal{A}$, and $y,z \in \mathcal{N}_{r(\mathcal{A})}^+$, then $\vartheta(yz) = \vartheta(y)\vartheta(z)$. Here, $\mathcal{N}_{r(\mathcal{A})}$ denotes the von Neumann algebra generated by $r(\mathcal{A})$. 
        \item \label{functorringiso} If $\tau_2(r(p)) = \tau_1(p)$ for all $p \in \mathcal{A}$ and $\sup_{p \in \mathcal{A}} \{ r(p) \} = 1_{\mathcal{N}}$, the restriction of $\vartheta$ to $\mathcal{N}_{r(A)}$ is a trace preserving $*$-algebra homomorphism between $\mathcal{N}_{r(A)}$ and $\mathcal{M}_{\mathcal{A}}$ 
    \end{enumerate}
\end{thm}

\begin{proof}
    Let $\Phi: \mathcal{M}_{\mathcal{A}} \to L^{\infty}(X,\Sigma,\theta)$ be the trace preserving $*$-algebra isomorphism given by Theorem \ref{abeliano}.
    
    First, we define the map $\vartheta$ over $\mathcal{N}^+$. Fix $y \in \mathcal{N}^+$ and let $\rho_y$ be the measure on the $\sigma$-ring generated by $\Phi(\mathcal{A})$ from Lemma \ref{inducedpremeasure}.  

    By the Radon-Nikodym theorem, there exists a nonnegative $\sigma(\Phi(\mathcal{A}))$-measurable function $h_y$ such that
    $$
\rho_y(E) = \int\limits_{E} h_y \, d\theta, \quad \forall E \in \sigma(\Phi(\mathcal{A})). 
    $$
    Set $\vartheta(y) := \Phi^{-1}(h_y)$. This defines a map $\vartheta: \mathcal{N}^+ \to \mathcal{M}_{\mathcal{A}}^+$. We will show that this map can be uniquely extended to a linear map defined on $\mathcal{N}$.

    Let $\alpha, \beta \in [0,\infty)$, $y,x \in \mathcal{N}^+$, $E \in S^+$ and set $z = \alpha x + \beta y$.  By Lemma \ref{inducedpremeasure}, for any $(B \setminus A) \cup C \in S^+$ we have
    \begin{align*}
\rho_z\Big( \big( B \setminus A\big) \cup C \Big) &= \tau_2\Big( z \big(r(q) - r(p) \big) \Big) = \tau_2\Big( (\alpha x + \beta y) \big(r(q) - r(p) \big) \Big) \\
&= \alpha \tau_2\Big( x \big(r(q) - r(p) \big) \Big) + \beta \tau_2\Big( y \big(r(q) - r(p) \big) \Big) \\
&= \alpha \rho_x\Big( \big( B \setminus A\big) \cup C \Big) + \beta \rho_y\Big( \big( B \setminus A\big) \cup C \Big).
    \end{align*}
    By uniqueness of the Caratheodory-Hahn constuction, it follows that $\rho_z = \alpha \rho_x + \beta \rho_y$. Uniqueness of the Radon-Nikodym derivative for $\sigma$-finite measures shows that $h_z = \alpha h_x + \beta h_y$. An application of $\Phi^{-1}$ shows that $\vartheta(z) = \alpha \vartheta(x) + \beta \vartheta(y)$. This shows that $\vartheta$ is linear on the convex cone $\mathcal{N}^+$.

    Since the traces $\tau_1$ and $\tau_2$ are faithful, it immediately follows that $\mathcal{N}_h$ and ${\mathcal{M}_{\mathcal{A}}}_h$ are Archimedean $\mathbb{R}$-vector lattices. Then, by \cite[Theorem~1.10]{positive} there is a unique $\mathbb{R}$-linear map from $\mathcal{N}_h \to {\mathcal{M}_{\mathcal{A}}}_h$ that extends $\vartheta$.

    We now extend $\vartheta$ to any operator in $\mathcal{N}$. Fix $y \in \mathcal{N}$ and define
    $$
\vartheta(y) = \frac{1}{2} \vartheta \Big(y + y^* \Big) - \frac{i}{2} \vartheta \Big( (iy) + (iy)^*  \Big). 
    $$
    Notice that $(y + y^*)$ and $\big( (iy) + (iy)^* \big)$ are self adjoint so their images under $\vartheta$ are defined. Also, if $y$ is self adjoint $\frac{1}{2} (y + y^*) = y$ and $\big( (iy) + (iy)^* \big) = 0$, so $\vartheta$ remains well defined on $\mathcal{N}_h$.

    We now show that $\vartheta$ is a $\mathbb{C}$-linear map. Let $\alpha = a + ib$ for $a,b \in \mathbb{R}$ and $x,y \in \mathcal{N}$. Then,
    \begin{align*}
        \vartheta(\alpha y) &= \frac{1}{2} \vartheta \Big(  ( \alpha y + \overline{\alpha} y^*) \Big) - \frac{i}{2} \vartheta \Big( (i\alpha y + \overline{i\alpha} y^*) \Big) \\
        &= \frac{1}{2} \vartheta \Big( a (y + y^*) + b\big( (iy) + (iy)^* \big) \Big) \\
        &- \frac{i}{2} \vartheta \Big( a \big( (iy) + (iy)^* \big) -b \big( y + y^* \big) \Big)\\
        &= \frac{a}{2} \vartheta (y + y^*) + \frac{b}{2} \vartheta \big( (iy)+(iy)^* \big) - i\frac{a}{2} \vartheta \big( (iy)+(iy)^* \big) + i\frac{b}{2} \vartheta (y + y^*) \\
        &= (a+ib)\frac{1}{2} \vartheta (y + y^*) - (a+ib) \frac{i}{2}\vartheta \big( (iy)+(iy)^* \big) \\
        &= \alpha \vartheta(y).
    \end{align*}
And
    \begin{align*}
        \vartheta(x + y) &= \frac{1}{2} \vartheta \Big(  ( x+y) + (x+y)^* \Big) - \frac{i}{2} \vartheta \Big( (ix+iy) + (ix+iy)^* \Big) \\
        &= \frac{1}{2} \vartheta \Big(  ( x+x^*) + (y+y^*) \Big) - \frac{i}{2} \vartheta \Big( \big(ix+ (ix)^* \big) + \big(iy+ (iy)^* \big) \Big) \\
        &= \frac{1}{2} \vartheta \Big(x + x^* \Big) + \frac{1}{2} \vartheta \Big(y + y^* \Big) - \frac{i}{2} \vartheta \Big( ix+ (ix)^*  \Big) -\frac{i}{2} \vartheta \Big( iy+ (iy)^* \Big)\\
        &= \vartheta(x) + \vartheta(y).
    \end{align*}
    This shows that $\vartheta$ is a $\mathbb{C}$-linear map.

    Since 
    \begin{align*}
\vartheta(y^*) &= \frac{1}{2} \vartheta \Big(y^* + y^{**} \Big) - \frac{i}{2} \vartheta \Big( (iy^*) + (iy^*)^*  \Big) \\
&= \frac{1}{2} \vartheta \Big(y + y^{*} \Big) - \frac{i}{2} \vartheta \Big( -(iy)^* - (iy)  \Big) \\
&= \frac{1}{2} \vartheta \Big(y + y^{*} \Big) - \frac{\overline{i}}{2} \vartheta \Big( (iy)^* + (iy)  \Big) = \big(\vartheta(y)\big)^*,
    \end{align*}
    we conclude that $\vartheta$ is a $^*$-preserving map.

    Statement \ref{functorpositive} follows immediately from the definition of $\vartheta$. To show statement \ref{functorlinf}, fix $y \in \mathcal{N}^+$, and let $h_y = \Phi(\vartheta(y))$. Let $\{p_n\} \in \mathcal{A}$ such that $p_n \uparrow 1_{\mathcal{M}}$ and $A_n \in \sigma(\Phi(\mathcal{A}))$ such that $\chi_{A_n} = \Phi(p_n)$.
    
    Define the sets 
    $$
E_{n,m} = \left\{ s \in A_{m}: h_y(s) > c\|y\|_{B(H_2)} + \frac{1}{n}  \right\}.
    $$
    Since $h_y$ is $\sigma(\Phi(\mathcal{A}))$-measurable, for every $\delta > 0$, there exists a sequences $\{p_k\}$, $\{q_k\}$ in $\mathcal{A}$ and a set $C$ such that $\theta(C) = 0$, 
    $$
E_{n,m} \subseteq C \cup \bigcup_{k=1}^\infty (B_k \setminus A_k) = G_{\delta},  
    $$
    satisfying $\theta(G_\delta) + \delta < \theta(E_{n,m})$,
    and $\chi_{B_k} = \Phi(q_k)$, $\chi_{A_k} = \Phi(p_k)$ for each $k \in \mathbb{N}$.

    Integration yields
    \begin{align*}        
 \Big(c\|y\|_{B(H_2)} + \frac{1}{n}\Big) \theta(E_{n,m}) &= \int\limits_{E_{n,m}} h_y \, d\theta \leq \rho_y(G_\delta) \leq \sum_{k=1}^\infty \rho_y(B_k \setminus A_k) \\
 &= \sum_{k=1}^\infty  \tau_2\Big(y \big( r(q_k) - r(p_k) \big)\Big) \\
 &\leq \|y\|_{B(H_2)}\sum_{k=1}^\infty    \tau_2\Big( r(q_k) - r(p_k)  \big)\Big)\\
 &\leq c \|y\|_{B(H_2)} \sum_{k=1}^\infty \tau_1(q_k - p_k) = c\|y\|_{B(H_2)} \theta(G_\delta)\\
 &\leq c\|y\|_{B(H_2)} \Big(\delta + \theta(E_{n,m}) \Big)
    \end{align*}
    Letting $\delta \downarrow 0$ yields a contradiction unless $\theta(E_{n,m}) = 0$. Taking limits when $n,m \to \infty$ shows that $h_y \leq c\|y\|_{B(H_2)}$ $\theta$-almost everywhere. An application of $\Phi^-1$ shows that $\|\vartheta(y)\|_{B(H_1)} \leq c\|y\|_{B(H_2)}$ and proves statement \ref{functorlinf}. 

To show statement \ref{functorincreasing}, suppose that $y_n \uparrow y$. Since $\vartheta$ is a positive linear operator $\vartheta(y_n)$ is increasing and $\vartheta(y_n) \leq \vartheta(y)$. Therefore, $\vartheta(y_n) \uparrow z \leq \vartheta(y)$ for some $z \in \mathcal{M}_{\mathcal{A}}^+$. Let $\{p_m\} \in \mathcal{A}$ satisfy $p_m \uparrow 1$ and $\{U_m\} \in \Sigma$ satisfying $\chi_{U_m} = \Phi(p_m)$, then
\begin{align*}
    \tau_1(p_m \vartheta(y_n) p_m) &= \tau_1(\vartheta(y_n) p_m) = \int_{U_m} h_{y_n} \, d\theta = \tau_2(y_n r(p_m)) \\
    &= \tau_2(r(p_m) y_n r(p_m)) 
\end{align*}
Letting $n \to \infty$ yields $$\tau_1(p_m z \, p_m) = \int_{U_m} z \, d\theta = \tau_2(r(p_m) y \, r(p_m)) = \tau_1(p_m \vartheta(y) p_m).$$ 
From Lemma \ref{lemaigualador}, it follows that $p_m z \, p_m = p_m \vartheta(y) \, p_m$ and letting $m \to \infty$ shows that $\vartheta(y) = z$ completing the proof of (iii).

To prove \ref{functorlinear}, notice that by construction, the equality (\ref{taufunctor}) holds for $y \in \mathcal{N}^+$. Linearity of the map $\vartheta$ and of the traces $\tau_1,\tau_2$ extends the equality to all $y \in \mathcal{N}$.

To prove \ref{functorcoreproj}, let $p \in \mathcal{A}$ such that $\tau_1(p) = \tau_2(r(p))$. By \ref{functorlinf}, we have $$\|\vartheta(r(p))\|_{B(H_1)} \leq c \|r(p)\|_{B(H_2)} = 1$$. 

An application of the formula (\ref{taufunctor}) yields
$$
\tau_1(p \, \vartheta(r(p))) = \tau_2\Big( \big(r(p)-r(0)\big) r(p) \Big) = \tau_2(r(p)) = \tau_1(p). 
$$

Since $\vartheta(r(p))$ and $p$ commute and are positive, we get 
$$
p \, \vartheta(r(p)) = \vartheta(r(p)) p \leq \|\vartheta\big( r(p)\big)\|_{B(H_1)} p \leq p. 
$$
Since $\tau_1$ is faithful, it follows that $p \,  \vartheta(r(p)) = p$, thus $p \vartheta(r(p))p = p$. Let $q \in \mathcal{A}$, using formula (\ref{taufunctor}) we get
$$
\tau_1\Big(q \, \vartheta\big( r(p)\big) q\Big) = \tau_1\Big(q \, \vartheta\big( r(p)\big)\Big) = \tau_2\Big(r(q)  r(p)\Big) = \tau_2(r(p \wedge q)).
$$
Letting $q \uparrow 1_{\mathcal{M}}$ yields $\tau_1\big(\vartheta(r(p))\big) = \tau_2(r(p))$. Since $\tau_1(p) = \tau_2(r(p))$, it follows that $\tau_1\big(\vartheta(r(p))\big) = \tau_1\big(p \, \vartheta(r(p) p )\big)$. Notice that $p \, \vartheta(r(p) p \leq \vartheta(p)$, then by normality of $\tau_1$ we get that $p \, \vartheta(r(p) p = \vartheta(p)$. It follows that 
$$p = p \, \vartheta(p) p  = \vartheta(p),$$ 
and proves statement \ref{functorcoreproj}.

To prove \ref{functormult}, fix $p \in \mathcal{A}$, $y,z \in \mathcal{N}_{r(\mathcal{A})}^+$ and let $\psi: \mathcal{N}_{r(A)} \to L^{\infty}(Y,\Sigma_2,\theta_2)$ be the trace preserving $*$-algebra isomorphism provided by Theorem \ref{abeliano}. Define the maps $\zeta,\eta: \Sigma_2 \to [0,\infty]$ as follows
$$
\zeta(E) = \tau_1\big(p \, \vartheta(y) \vartheta(\psi^{-1}(\chi_{E})) \big) \quad \text{and} \quad  \eta(E) = \tau_1\big(p \, \vartheta(y \, \psi^{-1}(\chi_{E})) \big),
$$
hold for all $E \in \Sigma_2$.

Since $\|\psi^{-1}(\chi_E)\|_{B(H_2)} \leq 1$, by \ref{functorlinf}, and the fact that $p$,$\vartheta(y)$,$\vartheta(\psi^{-1}(\chi_E))$, and $\vartheta(y \psi^{-1}(\chi_E))$ commute, we conclude that both functions are bounded above by $\|y\|_{B(H_2)}\tau_1(p) < \infty$, it is clear that $\zeta(\emptyset) = \nu(\emptyset) = 0$. To show that the functions define finite measures, let $E = \cup_{k} E_k$ be a disjoint union in $\Sigma_2$ and notice that $\chi_E = \sup_n \sum_{k=1}^n \chi_{E_k}$. By \ref{functorincreasing} and normality of the trace, we get
\begin{align*}
    \zeta(E) &= \tau_1\big(p \, \vartheta(y) \vartheta(\psi^{-1}(\sup_{n} \sum_{k=1}^n \chi_{E_k})) \big) = \sup_n \sum_{k=1}^n \tau_1\big( p \, \vartheta(y) \vartheta(\psi^{-1}(E_k))\big) \\
    &= \sup_n \sum_{k=1}^n \zeta(E_k) = \sum_{k=1}^\infty \zeta(E_k).
\end{align*}
A similar argument shows that $\eta(E) = \sum_{k=1}^\infty \eta(E_k)$ and shows that both $\zeta$ and $\eta$ are finite measures over $\Sigma_2$.

Consider the collection 
$$
V = \left\{ E \in \Sigma_2: \chi_E = \psi(r(q)) \text{ for some } q \in \mathcal{A} \right\}.
$$
For each $E \in V$, using \ref{functorcoreproj} and the fact that $r$ is order preserving, we have
\begin{align*}
\zeta(E) &= \tau_1\big( p \, \vartheta(y) \vartheta(\psi^{-1}(E))\big) = \tau_1\big( p \, \vartheta(y) \vartheta(\psi^{-1}(\psi(r(q))))\big) \\
&= \tau_1\big( p \, \vartheta(y) \vartheta(r(q)) \big) = \tau_1\big( p \, \vartheta(y) q \big)\\
&= \tau_1\big((p \wedge q) \vartheta(y)\big) = \tau_2\big( r(p \wedge q) y \big),
\end{align*}
and
\begin{align*}
\eta(E) &= \tau_1\big( p \, \vartheta(y \, \psi^{-1}(E))\big) = \tau_1\big( p \, \vartheta(y \, \psi^{-1}(\psi(r(q))))\big)\\
&= \tau_1\big( p \, \vartheta(y \, r(q)) \big)\\
&= \tau_2\big( r(p) y \, r(q) \big) =  \tau_2\big( r(p \wedge q) y \big).
\end{align*}

If $S^+_{r(\mathcal{A})}$ is the semiring provided by Lemma \ref{coremedida} associated to the ordered core $r(\mathcal{A})$, taking differences the measures $\eta$ and $\zeta$ coincide on the semiring $S^+_{r(\mathcal{A})}$, therefore they coincide on the generated $\sigma$-ring. The space $L^{\infty}(Y,\sigma(S^+_{r(\mathcal{A})}),\theta_2)$ is a von Neumann subalgebra of $L^{\infty}(Y,\Sigma_2,\theta_2)$ containing $\psi(r(\mathcal{A}))$, its image by $\psi^{-1}$ is a von Neumann subalgebra of $\mathcal{N}_{r(\mathcal{A})}$ containing $r(\mathcal{A})$, therefore they coincide. This shows that $L^{\infty}(Y,\sigma(S^+_{r(\mathcal{A})}),\theta_2)$ coincides with $L^{\infty}(Y,\Sigma_2,\theta_2)$, thus $\zeta$ and $\eta$ define the same measure over $\Sigma_2$.

Therefore,
$$
\tau_1\big(p \, \vartheta(y) \vartheta(\psi^{-1}(\chi_E))\big) = \tau_1\big(p \, \vartheta(y \, \psi^{-1}(\chi_E))\big), \quad \forall E \in \Sigma_2. 
$$
An application of $\Phi$ yields
$$
\int_{A_p} \Phi(\vartheta(y) \vartheta(\psi^{-1}(\chi_E))) \, d\theta = \int_{A_p} \Phi(\vartheta(y \, \psi^{-1}(\chi_E))) \, d\theta = \rho_{y \psi^{-1}(\chi_E)}(A_p),
$$
here $\chi_{A_p} = \Phi(p)$. By the uniqueness $\theta$-a.e. of the Radon-Nikodym derivative we get $$\Phi(\vartheta(y) \vartheta(\psi^{-1}(\chi_E))) = \Phi(\vartheta(y \, \psi^{-1}(\chi_E))),$$ therefore $\vartheta(y)\vartheta(\psi^{-1}(\chi_E)) = \vartheta(y \, \psi^{-1}(\chi_E))$.

Since $z \in \mathcal{N}_{r(\mathcal{A})}^+$, the nonnegative function $\psi(z)$ is the increasing limit of finite positive linear combinations of characteristic functions of sets in $\Sigma_2$. So by \ref{functorincreasing}, we get 
$$
\vartheta(y)\vartheta(\psi^{-1}(\psi(z))) = \vartheta(y \, \psi^{-1}(\psi(z))),
$$
thus $\vartheta(y)\vartheta(z) = \vartheta(yz)$ and completes the proof of \ref{functormult}.

To prove \ref{functorringiso}, we have already shown that $\vartheta$ is a $*$-preserving linear map; it remains to show that it preserves multiplication and traces. Let $y,z \in \mathcal{N}_{r(\mathcal{A})}$, we can decompose these operators into linear combinations of positive operators as follows 
\begin{align*}
y &= \frac{1}{2}\Big( \big(y + y^*\big)^+ - \big(y + y^*\big)^- \Big) - \frac{i}{2}\Big( \big(iy + (iy)^*\big)^+ - \big(iy + (iy)^*\big)^- \Big) \\
z &= \frac{1}{2}\Big( \big(z + z^*\big)^+ - \big(z + z^*\big)^- \Big) - \frac{i}{2}\Big( \big(iz + (iz)^*\big)^+ - \big(iz + (iz)^*\big)^- \Big).
\end{align*}
For an operator $x$, denote $(x)^{\mathfrak{d}(1)} = (x)^+$ and $(x)^{\mathfrak{d}(-1)} = (x)^{-}$. Linearity of $\vartheta$ and \ref{functormult} yield
\begin{align*}
    \vartheta(yz) &= \frac{1}{4} \sum_{j,k =0}^1 (-1)^{j+k} \vartheta\Big(\big( y+y^* \big)^{\mathfrak{d}(j)}\Big) \vartheta\Big(\big( z+z^* \big)^{\mathfrak{d}(j)} \Big)\\ 
    &+ \frac{1}{4} \sum_{j,k =0}^1 (-1)^{j+k} \vartheta\Big(\big( (iy)+(iy)^* \big)^{\mathfrak{d}(j)}\Big) \vartheta\Big(\big( (iz)+(iz)^* \big)^{\mathfrak{d}(j)} \Big)\\
    &= \vartheta(y)\vartheta(z).
\end{align*}
This shows that $\vartheta$ preserves multiplication. Finally, fix $y \in \mathcal{N}_{r(\mathcal{A})}^+$ and let $p_{n} \uparrow 1_{\mathcal{M}}$ with $\{p_n\} \in \mathcal{A}$, then by \ref{taufunctor} we get
$$
\tau_1(p_n \vartheta(y)) = \tau_2(r(p_n)y), \quad \forall n \in \mathbb{N},
$$
taking limits when $n \to \infty$ yields $\tau_1(\vartheta(y)) = \tau_2(y)$. Linearity of the trace extends this equality for all elements $y \in \mathcal{N}_{r(\mathcal{A})}^+$. Thus, $\vartheta$ preserves traces and completes the proof.

\end{proof}

\section{A measure space tailored to an ordered core}
In this section, we assume that $(\mathcal{M},\tau)$ is a von Neumann algebra with a semifinite normal faithful trace and $\mathcal{A}$ a $\sigma$-bounded full ordered core with maximal extension $\mathcal{A}_M$. The main result of this section is that we construct transition maps between the von Neumann algebra $\mathcal{M}$ and a space $L^{\infty}_\lambda$ for a distinguished Borel measure on the half line. This is proved in Theorem \ref{QRmaps}. 

Our next task is to induce a Borel measure on the half line. To do so, we define the following functions over $[0,\infty)$:
$$
a(s) = \sup  \Big([0,s] \cap \Gamma_{\mathcal{A}_M} \Big) \quad \text{and} \quad b(s) = \inf \Big( [x,\infty) \cap \Gamma_{\mathcal{A}_M} \Big),
$$
where the infimum of the empty set is taken as $+\infty$. It is immediate that $a$ and $b$ are non-decreasing functions, so they are Borel measurable. More properties of these functions are explored in the following proposition.
\begin{prop}\label{abmaps}
Let $\mathcal{A}_M$ be a maximal ordered core, then $a(s) \leq b(s)$. Also, the following are equivalent.
\begin{enumerate}
    \item $a(s) = b(s)$.
    \item $s \in \Gamma_{\mathcal{A}_M}$.
\end{enumerate}
\end{prop}
\begin{proof}
    If $\mathcal{A}$ is maximal then $\Gamma_{\mathcal{A}_M}$ is a closed subset of $[0,\infty)$. Thus for each $s \in [0,\infty)$ we have that $a(s) \in [0,x] \cap \Gamma_{\mathcal{A}_M}$ and $b(s) \in [s,\infty) \cap \Gamma_{\mathcal{A}_M}$, hence $a(s) \leq b(s)$ with equality only possible if $a(s) = b(s) = s$ and $s \in \Gamma_{\mathcal{A}_M}$. 
    
\end{proof}

Since $b$ is non-decreasing, it is Borel measurable. We define the measure $\lambda$ as the pushforward measure of the function $b$, that is
\begin{equation}\label{pushforward1}
\lambda(E) = m(b^{-1}(E)), \quad \text{for all } E \subseteq [0,\infty] \text{ Borel measurable},    
\end{equation}
and 
\begin{equation}\label{pushforward2}
\int\limits_{[0,\infty)} \varphi \, d\lambda = \int\limits_{\Gamma_{\mathcal{A}_M}} \varphi \, d\lambda = \int\limits_{[0,\infty]} \varphi \circ b \, dm,
\end{equation}
whenever the function $\varphi \circ b$ is Lebesgue measurable. Here $m$ denotes the Lebesgue measure.

The measure space $([0,\infty),\lambda)$ encodes the monotonicity properties of the ordered core $\mathcal{A}$. We will explore this relationship in the next two propositions. The first proposition induced a map between $\mathcal{M} \to L^{^\infty}_\lambda$, the next one defines a partial inverse. 

\begin{prop}\label{Rmap}
    Let $(\mathcal{M},\tau)$ be a von Neumann algebra with a semifinite normal faithful trace $\tau$ and $\mathcal{A}$ be a $\sigma$-bounded full ordered core. There exists a $*$-preserving positive linear map $R: \mathcal{M} \to L^{\infty}_\lambda$ such that:
     \begin{enumerate}[label=(\roman*)]
        \item \label{Rincreasing} If $\{x_n\}_n \in \mathcal{M}^+$ and $x \in \mathcal{M}^+$ such that $x_n \uparrow x$, then $R(x_n) \uparrow R(x)$ $\lambda$-almost everywhere.
        \item \label{Rintegrals} For each $p \in \mathcal{A}$, $\tau(xp) = \int\limits_{[0,\tau(p)]} R(x) \, d\lambda$, for all $x \in \mathcal{M}^+$.
        \item \label{Rcoresets} For each $p \in \mathcal{A}$, $R(p) = \chi_{[0,\tau(p)]}$.
        \item \label{Rpositive} If $0 \leq x \leq y \in \mathcal{M}^+$, then $0 \leq R(x) \leq R(y)$  $\lambda$-almost everywhere.
        \item \label{Ralgebrahomo} The restriction of $R$ to $\mathcal{M}_{\mathcal{A}}$ is a $*$-algebra homomorphism between $\mathcal{M}_{\mathcal{A}} \to L^{\infty}_{\lambda}$ that is trace preserving. 
        \end{enumerate}
\end{prop}
\begin{proof}
Consider the abelian von Neumann algebra $L^{\infty}_\lambda([0,\tau(1_{\mathcal{M}})])$ and the ordered core 
$$
\mathcal{B} = \{ \chi_{[0,\tau(p)]}: p \in \mathcal{A} \}.
$$
Define the map $r: \mathcal{B} \to \mathcal{A}$ by $r(\chi_{[0,\tau(p)]}) = p$. Notice that the mapping is well defined since $\tau$ is injective over $\mathcal{A}$. Notice that
$$
\int\limits_{[0,\infty)} \chi_{[0,\tau(p)]} \, d\lambda  = \lambda([0,\tau(p)]) = \tau(p),
$$
Therefore, $r$ is a trace preserving core morphism, and the operator $R$ is the map provided by Theorem \ref{functor}. All the statements \ref{Rincreasing}-\ref{Ralgebrahomo} follow immediately from their corresponding statements on Theorem \ref{functor}. 

\end{proof}

The following proposition studies a partial inverse of $R$.

\begin{prop}\label{Qmap}
    Let $(\mathcal{M},\tau)$ be a von Neumann algebra with a semifinite normal faithful trace $\tau$ and $\mathcal{A}$ be an ordered core. There exists a $*$-algebra homomorphism $Q: L^{\infty}_\lambda \to \mathcal{M}_{\mathcal{A}}$ such that:
    
     \begin{enumerate}[label=(\roman*)]
        \item \label{Qincreasing} If $\{f_n\}_n \in L^{\infty}_{\lambda}$ and $f \in L^+_\lambda$ such that $f_n \uparrow f$, then $Q(f_n) \uparrow Q(f)$.
        \item \label{Qintegrals} For each $p \in \mathcal{A}$, $\tau\big(Q(f)p\big) = \int\limits_{[0,\tau(p)]} f \, d\lambda$, for all $f \in L^+_\lambda$.
        \item \label{Qcoresets} For each $p \in \mathcal{A}$, $p = Q(\chi_{[0,\tau(p)]})$.
        \item \label{Qpositive} If $0 \leq f \leq g$ $\lambda$-almost everywhere, then $0 \leq Q(f) \leq Q(g)$. 
    \end{enumerate}
\end{prop}
\begin{proof}
Consider the ordered core 
$$
\mathcal{B} = \{ \chi_{[0,\tau(p)]}: p \in \mathcal{A} \}.
$$
Define the map $r: \mathcal{A} \to \mathcal{B}$ by $r(p) = \chi_{[0,\tau(p)]}$. Notice that
$$
\int\limits_{[0,\infty)} \chi_{[0,\tau(p)]} \, d\lambda  = \lambda([0,\tau(p)]) = \tau(p),
$$
Therefore, $r$ is a trace preserving core morphism, hence Theorem \ref{functor} provides an operator Q defined from the measure space $L^{\infty}([0,\infty],\sigma(\mathcal{B}),\lambda) \to \mathcal{M}_{\mathcal{A}}$. Since the complement $\Gamma_{\mathcal{A}}$ has zero $\lambda$-measure, we conclude that the domain of $Q$ is $L^{\infty}_\lambda$. All the statements \ref{Qincreasing}-\ref{Qpositive} follow immediately from their corresponding statements on Theorem \ref{functor}. 
    
\end{proof}

The close connection between the maps $R$ and $Q$ is shown in the following theorem.
\begin{thm}\label{QRmaps}
    Let $(\mathcal{M},\tau)$ be a von Neumann algebra with a semifinite normal faithful trace $\tau$, $\mathcal{A}$ be a $\sigma$-bounded full ordered core, and the maps $R,Q$ given by Propositions \ref{Rmap} and \ref{Qmap}. Then:
        \begin{enumerate}[label=(\roman*)]
        \item \label{RQinverse} The equality $RQ(f) = f$ holds $\lambda$-almost everywhere, for all $f \in (L^{\infty}_\lambda)^+$.
        \item \label{QRinverse} The equality $QR(x) = x$ holds, for all $x \in (\mathcal{M}_{\mathcal{A}})^+$.
        \end{enumerate}

        In addition,  $R$ and $Q$ extend uniquely to a normal $*$-algebra isomorphism between $S(\tau_{\mathcal{A}}) \to S(\lambda)$. Here $S(\tau_{\mathcal{M}_\mathcal{A}})$ denotes the $\tau$-measurable operators affiliated with $\mathcal{M}_\mathcal{A}$ and 
$$
S(\lambda) = \{ f \in L^{0}_\lambda: \exists E \in [0,\infty) \text{ such that } (1-\chi_E) f  \in L^{\infty}_\lambda \text{ and } \lambda(E) < \infty \}.
$$
Moreover, 
        \begin{enumerate}[label=(\roman*)]\addtocounter{enumi}{2}
        \item \label{QRfunctor} For all $x \in S(\tau_{\mathcal{M}_{\mathcal{A}}})^+$ and $f \in S(\lambda)^+$ the formulas
        \begin{equation}\label{extfunctor}
            \tau\big(x p\big) = \int\limits_{[0,\tau(p)]} R(x) \, d\lambda \quad \text{and} \quad \tau\big(Q(f) p\big) = \int\limits_{[0,\tau(p)]} f \, d\lambda
        \end{equation}
        hold for all $p \in \mathcal{A}$.
        \item \label{Rcoredec} If $x \in \mathcal{M}^{\downarrow}_{\mathcal{A}}$, then $Rx$ is a $\lambda$-measurable decreasing function.
        \item \label{Qcoredec} If $g \in S(\lambda)^+ $ is a decreasing function, then $g = R(x)$ for a unique $x \in \mathcal{M}^{\downarrow}_{\mathcal{A}}$.
        \item \label{Rrearrangement} For each $x \in S(\tau_{\mathcal{M}_\mathcal{A}})$, $\mu(t;x) = \mu(t;R(x))$. Here, the rearrangement of $R(x)$ is taken with respect to the measure $\lambda$.
        \item \label{Qrearrangement} For each $f \in S(\lambda)$, $\mu(t;Qf) = \mu(t;f)$. Here, the rearrangement of $f$ is taken with respect to the measure $\lambda$.
    \end{enumerate}
\end{thm}
\begin{proof}
Fix $f \in L^{+}_\lambda$ and $p \in \mathcal{A}$. By Proposition \ref{Qmap}.\ref{Qintegrals} and Proposition \ref{Rmap}.\ref{Rintegrals} we get
$$
\int\limits_{[0,\tau(p)]} RQ(f) \, d\lambda = \tau\big((Qf) p \big) = \int\limits_{[0,\tau(p)]} f \, d\lambda. 
$$
Since $f$ and $RQ(f)$ are $\sigma(\mathcal{B})$-measurable, here $\mathcal{B} = \{ \chi_{[0,\tau(p)]} : p \in \mathcal{A} \}$, we conclude that $f = RQ(f)$ and proves \ref{RQinverse}.

Fix $x \in \mathcal{M}_{\mathcal{A}}^+$ and $p \in \mathcal{A}$. By Proposition \ref{Qmap}.\ref{Qintegrals} and Proposition \ref{Rmap}.\ref{Rintegrals} we get
$$
\tau\big(QR(x)p\big) =
\int\limits_{[0,\tau(p)]} R(x) \, d\lambda = \tau\big(x p \big), \forall p \in \mathcal{A}. 
$$
An application of Lemma \ref{lemaigualador} shows that $QR(x) = x$ and proves \ref{QRinverse}. 

Since $RQ$ and $QR$ coincide with the identity maps on positive elements and $Q,R$ are $*$-algebra homomorphisms, it follows that $R$ and $Q$ are trace preserving $*$-algebra isomorphisms and hence normal. By \cite[Proposition~2.9.2]{dodsbook} the map $R$ extends to a unique $*$-algebra trace preserving isomorphism from $S(\tau_{\mathcal{M}_\mathcal{A}}) \to S(\lambda)$ and $Q$ extends uniquely to its inverse. 

To prove \ref{QRfunctor}, fix $x \in S(\tau_{\mathcal{M}_\mathcal{A}})$ and $f \in S(\lambda)^+$. Choose sequences $\{x_n\} \in M^+$ and $\{f_n\} \in (L^{\infty}_\lambda)^+$ such that $x_n \uparrow x$ and $f_n \uparrow f$. Then \ref{QRfunctor} follows from normality of the traces, Proposition \ref{Rmap}.\ref{Rintegrals}, and Proposition \ref{Qmap}.\ref{Qintegrals}. 

By definition of $\mathcal{M}^\downarrow_{\mathcal{A}}$, if $x \in \mathcal{M}^\downarrow_{\mathcal{A}}$, there exists a sequence of $\{x_n\} \in \mathcal{M}$ such that 
$$
x_n = \sum_{k=1}^{m_{n}} \alpha_{n,k} p_{n,k},
$$
for some $p_{n,k}$ and $\alpha_{n,k} \in [0,\infty)$. By Proposition \ref{Rmap}.\ref{Rcoresets} we get
$$
R(x_n) = \sum_{k=1}^{m_{n}} \alpha_{n,k} R(p_{n,k}).
$$
The normality of $R$ implies that $R(x_n) \uparrow R(x)$. Since $R(x_n)$ is decreasing, it follows that $R(x)$ is decreasing and proves statement \ref{Rcoredec}.

Similarly, if $f \in S(\lambda)^+$ and is decreasing, there exists a sequence $\{f_n\}$ of functions of the form
$$
f_n = \sum_{k=1}^{m_{n}} \alpha_{n,k} \chi_{[0,\tau(p_{n,k})]},
$$
for some $p_{n,k}$ and $\alpha_{n,k} \in [0,\infty)$. By Proposition \ref{Qmap}.\ref{Qcoresets} and normality of $Q$, we get $Q(f_n) \uparrow Q(x) \in \mathcal{M}^{\downarrow}_{\mathcal{A}}$ and proves statement \ref{Qcoredec}.

Statements \ref{Rrearrangement} and \ref{Qrearrangement} follow immediately from \cite[Proposition~3.3.10]{dodsbook}.
    
\end{proof}

An immediate consequence of the previous theorem is that core decreasing operations are in bijection with decreasing nonnegative functions in $S(\lambda)$, finite almost everywhere.

We conclude this section with concrete examples of operators $ R$ and $ Q$. We start with an ordered core for a commutative von Neumann algebra, this example shows that the map $R$ need not preserve products.

\begin{ex}\label{exCommutative2}
    Following Example \ref{exCommutative}, consider the measure space $(\mathbb{R}^2,\Sigma,\theta)$ with $\Sigma$ the Borel $\sigma$-algebra, $\theta$ the Lebesgue measure, and the trace given by the Lebesgue integral. Let $\mathcal{A}$ be the core induced by the family of closed balls $B_r = \{ s \in \mathbb{R}^2: \|s\|_{2} \leq r \}$ for $r > 0$. Since $\Gamma_{\mathcal{A}} = [0,\infty)$, it follows that $a(s) = s = b(s)$, so the measure $\lambda$ is the Lebesgue measure on $[0,\infty)$.
    
    By construction of the maps $R$ and $Q$, they are completely determined by their action on positive elements. Fix $\varphi \in L^{\infty}_\lambda$, and define $g_{\varphi}(s) = \varphi(\pi \|s\|^2_{2}) $ for each $s \in \mathbb{R}^2$. Integration in polar coordinates yields 
    $$
\int_{B(r)} g_\varphi \, d\theta = \int_{0}^r \int_{0}^{2 \pi} \varphi(\pi \rho^2) \rho \, d\theta \, d\rho = \int_{0}^{r} \varphi(\pi \rho^2) \, 2\pi \rho \, d\rho = \int_{0}^{\theta(B_r)} \varphi \, d\lambda.  
    $$
    Since $g_\varphi \in \mathcal{M}_{\mathcal{A}}$ and its integral coincides with the integral of $Q\varphi$ over all sets $B_r$, by Lemma \ref{lemaigualador} we conclude that $Q \varphi = g_{\varphi}$. Since $\varphi$ was an arbitrary positive function, we conclude that the operator $Q$ is given by the formula $Q\varphi(s) = \varphi(\pi \|s\|_{2}^2)$.

    The construction of $R$ makes it unlikely to find an explicit formula for all functions in $L^{\infty}_\theta$, however, we show an example for a particular function. Let 
    $$
W = \{ \big(\rho \cos(\theta), \rho \sin(\theta)\big): 0 < \rho < \theta, \theta \in [0,\pi] \},
    $$
    and set $f = \chi_{W}$. For each $r > 0$, integration in polar coordinates yields
    \begin{align*}
\int_{B_{r}} f \, d\theta &= \min\Big( \pi^3/3, r^3/6 + r^2(\pi - r)/2 \Big) \\
&= \int_{0}^{\theta(B_r)} \Big( \frac{1 - t^{1/2} \pi^{-3/2}}{2}  \Big) \chi_{(0,\pi^3)}(t) \, dt. 
    \end{align*}
    Therefore, $Rf(t) = \frac{1}{2}(1 - t^{1/2} \pi^{-3/2})\chi_{(0,\pi^3)}(t)$. This shows that $R$ (acting on the whole algebra $\mathcal{M}$) needs not preserve idempotent elements, so it does not preserve products.
\end{ex}

The next example shows an explicit formula for both operators $R$ and $Q$.

\begin{ex}\label{exHilbert2}
    Consider the ordered core from Example \ref{exHilbert}. Since $\Gamma_{\mathcal{A}} = \mathbb{N}$, it follows that $a(s) = \lfloor s \rfloor$ and $b(s) = \lceil s \rceil$. 
    
    The induced measure $\lambda$ may be identified with the counting measure over $\mathbb{N}$. 
    
    Fix $x \in B(H)^+$, a straightforward computation shows that
    $$
\tau(x p_n) = \sum_{k=1}^{n} \langle x e_k,e_k \rangle = \int\limits_{[0,n]} \psi_x \, d\lambda,
    $$
    here $\psi_x(s) = \langle x e_{\lceil s \rceil}, e_{\lceil s \rceil}\rangle$. By Lemma \ref{lemaigualador} we get that $Rx = \psi_x$. Since the map $x \mapsto (s \mapsto \langle x e_{\lceil s \rceil}, e_{\lceil s \rceil}\rangle)$ is linear and coincides with $R$ on positive operators, this closed formula extends to all operators $x \in \mathcal{M}$.

A direct computation shows that the mapping $T : L^{\infty}_\lambda \to \mathcal{M}_{\mathcal{A}}$, given by the formula
$$
T \varphi(\xi) = \sum_{k=1}^\infty \varphi(k) \langle \xi , e_k \rangle e_k,
$$
satisfies $RT \varphi = \varphi$ $\lambda$-almost everywhere, since $R$ restricts to a $*$-isomorphism on $\mathcal{M}_{\mathcal{A}}$, it follows that $Q = T$. 
\end{ex}

\section{Non-commutative Down Spaces}
Let $S({\mathcal{M}})$ be the set of $\mathcal{M}$-measurable operators associated to the von Neumann algebra $\mathcal{M}$ equipped with a semifinite normal faithful trace $\tau$ and a full $\sigma$-bounded ordered core $\mathcal{A}$. 

In this section, we give our extension of the level function to a noncommutative setting. This is done in Definition \ref{deflevel}. Using this analogue for the level function, we introduce our down spaces and prove that they are a normed space in Theorem \ref{downisnormed}.

We will focus on a subcollection of $S(\mathcal{M})$ taking the role of locally integrable functions.

\begin{defn}\label{L1loc}
    We say that an operator $x \in S(\mathcal{M})$ belongs to the space $L^{1}_{\mathcal{A},\text{Loc}}(\tau)$ if
    $$
\tau(\left|xp\right|) < \infty, \quad \text{ and } \tau(\left|px\right|) < \infty, \quad \text{for all} \in \mathcal{A}.
    $$
\end{defn}

Notice that any $p \in \mathcal{A}$ has finite trace, therefore, the operators $xp$ and $px$ belong to $S_0(\tau)$. Hence, the traces are well defined. Some basic properties of this space are collected.
\begin{prop}\label{l1loc}
    For a von Neumann algebra $\mathcal{M}$ equipped with a semifinite normal faithful trace $\tau$ and a $\sigma$-bounded full ordered core $\mathcal{A}$. Let $L^{1}_{\mathcal{A},\text{Loc}}$ be defined as above. Then $L^{1}_{\mathcal{A},\text{Loc}}$ is a vector space over $\mathbb{C}$ containing $\mathcal{M}$ satisfying:
    \begin{itemize}
        \item  $x \in L^{1}_{\mathcal{A},\text{Loc}}$  if and only if $x^* \in L^{1}_{\mathcal{A},\text{Loc}}$.
        \item If $x \in L^{1}_{\mathcal{A},\text{Loc}}$, then $|x| \in L^{1}_{\mathcal{A},\text{Loc}}$.
        \item If $|x| \leq |y|$ and $|y| \in L^{1}_{\mathcal{A},\text{Loc}}$, then $|x| \in L^{1}_{\mathcal{A},\text{Loc}}$.
    \end{itemize}  
\end{prop}
\begin{proof}
    Let $x,y \in L^{1}_{\mathcal{A},\text{Loc}}$, $\alpha \in \mathbb{C}$ then for each $p \in \mathcal{A}$ 
    $$
\tau(\left| \alpha xp \right|) = \left| \alpha \right| \tau(\left|xp\right|) < \infty \quad \text{ and} \quad \tau(\left|p\alpha x \right|) = \left| \alpha \right| \tau(\left|px \right|) < \infty,
    $$
    thus $\alpha x \in L^{1}_{\mathcal{A},\text{Loc}}$. By the polar decomposition, there exist partial isometries $u,v \in \mathcal{M}$ such that
    \begin{equation}\label{subaditive1}
\tau(\left| (x + y)p \right|) = \tau(u^*(x + y)p) = \tau(u^*xp) + \tau(u^*yp) \leq \tau(\left| xp\right|) + \tau(\left|yp\right|) < \infty.  
    \end{equation}
    and
\begin{equation}\label{subaditive2}
\tau(\left| p(x + y) \right|) = \tau(v^*p(x + y)) = \tau(u^*px) + \tau(u^*py) \leq \tau(\left| px\right|) + \tau(\left|py\right|) < \infty.  
    \end{equation}
    Above, we have used \cite[Proposition~3.4.5]{dodsbook}, together with the observation that $\|u\|_{B(H)} = \|v\|_{B(H)} \leq 1$ and $px,xp \in L^{1}(\tau)$. This shows that $x + y \in L^{1}_{\mathcal{A},\text{Loc}}$ and proves it is a $\mathbb{C}$-vector space.

    The fact that $L^{1}(\tau)$ is a $\mathcal{M}$-Banach bimodule of measurable operators shows that if $z \in \mathcal{M}$, then $\tau(|zp|) \leq \|z\|_{B(H)} \tau(p) < \infty$ and $\tau(|pz|) \leq \|z\|_{B(H)} \tau(p) < \infty$. Therefore $\mathcal{M} \subseteq L^{1}_{\mathcal{A},\text{Loc}}$.

    Also, 
    $$
\tau(|x^*p|) = \tau(\left| (x^*p)^* \right|) = \tau(\left| px \right|).
$$
An analogous argument shows that $\tau(|px^*|) = \tau(xp)$, thus $x \in L^{1}_{\mathcal{A},\text{Loc}}$ if and only if $x^* \in L^{1}_{\mathcal{A},\text{Loc}}$. 

Finally, notice that $$(|x|p)^*(|x|p)=p|x|^2p = (xp)^*(xp),$$
this implies that $||x|p| = |xp|$, thus $x \in L^{1}_{\mathcal{A},\text{Loc}}$ implies $|x| \in L^{1}_{\mathcal{A}}$. This completes the proof. 
    
\end{proof}

In general, an operator in $L^{1}_{\mathcal{A},\text{Loc}}$ can belong to $S(\mathcal{M}) \setminus S(\tau)$, as it is shown in the following example.
\begin{ex}\label{ejemplo}
    Let $\mathcal{M}$ be the commutative von Neumann algebra identified with the measure space $[0,\infty)$ with the Lebesgue measure and $\mathcal{A} = \{ \chi_{[0,y]}: y > 0 \}$ be an ordered core. Then, there exists an operator $x \in L^{1}_{\mathcal{A},\text{Loc}} \setminus S(\tau)$.
\end{ex}
\begin{proof}
Following Example \ref{exCommutative}, we identify this von Neumann algebra with $L^{\infty}$. It is straightforward to check that the space $L^{1}_{\mathcal{A},\text{Loc}}$ coincides with the usual definition of $L^{1}_\text{Loc}$. Hence, the function 
    $$
g = \sum_{k=1}^\infty k \chi_{[k,k+\frac{1}{k}]}
    $$
    is locally integrable. However, \cite[Example~2.3.13]{dodsbook} shows that
    $$
S(\tau) = \{ f \in L^{0}: \exists A \in \Sigma, |A^c| < \infty \text{ and } f\chi_A \in L^{\infty}\}.
    $$
    Suppose that $g \in S(\tau)$ seeking a contradiction. Then, there exists a measurable set $A$ such that $|A|^c < \infty$ and $g \chi_A \in L^{\infty}$. 
    
    Since $g \chi_{A}$ is bounded, there exists $N \in \mathbb{N}$ such that for all $k \geq N$ the intersection $A \cap [k,k+\frac{1}{k}]$ has zero measure. Therefore, up to a set of measure zero $\bigcup_{k \geq N} [k,k+\frac{1}{k}] \subseteq A^c$. This is absurd, as we supposed that $A^c$ has finite measure. Therefore $g \not\in S(\tau)$.
\end{proof}

The fact that for operators, the inequality $|x+y| \leq |x| + |y|$ does not hold, even for the simplest of noncommutative algebras, does not allow us to use the techniques from \cite[Section~7]{coredecreasing} to define the level function. Instead, we study a function induced by any positive measurable operator and define the level function based on it.

\begin{prop}\label{phiprop}
    Let $\mathcal{M},\tau,\mathcal{A}$ be as in the previous propositions. Let $\mathcal{A}_{M}$ and $\Gamma_{\mathcal{A}_{M}}$ be the maximal core from Theorem \ref{maximal} and its image by $\tau$ respectively. Let $a: \Gamma_{\mathcal{A}_{M}} \to [0,\infty)$ be the map from Proposition \ref{abmaps}. For each $x \in S^+ \cap L^{1}_{\mathcal{A},\text{Loc}}$ define the function $\varphi_{x}: [0,\infty) \to [0,\infty)$ by
    $$
\varphi_{x}(s) := \tau\big( \left| \tau_{\mathcal{A}_M}^{-1} (a(s)) x \right| \big). 
    $$
    Here $\tau_{\mathcal{A}_{M}}$ is the restriction of $\tau$ to $\mathcal{A}_M$. Then:
    \begin{enumerate}[label=(\roman*)]
     \item\label{phihomogeneous} For any $\alpha \in \mathbb{C}$, $\varphi_{\alpha x} = |\alpha| \varphi_x$.
     \item\label{philinear} For all $x,y \in  L^{1}_{\mathcal{A},\text{Loc}}$, $\varphi_{|x+y|}(s) \leq \varphi_{|x|}(s) + \varphi_{|y|}(s)$ for all $s \geq 0$.
     \item\label{phiincreasing} The function $\varphi_x$ is increasing and constant outside $\Gamma_{\mathcal{A}_M}$.
     \item\label{phiapprox} For all $x \in L^{1}_{\mathcal{A},\text{Loc}}$, there exists a sequence $x_n \in \mathcal{M}^+$ such that $x_n \uparrow |x|$ and $\varphi_{x_n}(s) \uparrow \varphi_{|x|}(s)$ for all $s \geq 0$.
 \end{enumerate} 
\end{prop}
\begin{proof}
    Since $\tau$ is injective when restricted to any ordered core, it follows that its inverse is well defined. Also, the fact that $x \in L^{1}_{\mathcal{A},\text{Loc}}$ ensures that $\varphi_x$ always takes finite values. Therefore, the definition of $\varphi_x$ makes sense. Statement \ref{phihomogeneous} is immediate. 
    
    To show statement \ref{philinear}. In the proof of Proposition \ref{l1loc} it was shown that $|x+y| \in L^{1}_{\mathcal{A},\text{Loc}}$ and using formula \ref{subaditive1}, we get
    \begin{align*}
        \varphi_{|x+y|}(s) &= \tau\big( \left| \tau_{\mathcal{A}_M}^{-1} (a(s)) |x+y| \right| \big) = \tau\big( \left| |x+y| \tau_{\mathcal{A}_M}^{-1} (a(s)) \right| \big) \\
        &= \tau\big( \left| (x+y) \tau_{\mathcal{A}_M}^{-1} (a(s)) \right| \big) \leq \tau\big( \left| x \,  \tau_{\mathcal{A}_M}^{-1} (a(s)) \right| \big) + \tau\big( \left| y \,  \tau_{\mathcal{A}_M}^{-1} (a(s)) \right| \big) \\
        &= \tau\big( \left| |x| \,  \tau_{\mathcal{A}_M}^{-1} (a(s)) \right| \big) + \tau\big( \left| |y| \,  \tau_{\mathcal{A}_M}^{-1} (a(s)) \right| \big) \\
        &= \tau\big( \left|  \tau_{\mathcal{A}_M}^{-1} (a(s)) |x| \right| \big) + \tau\big( \left|  \tau_{\mathcal{A}_M}^{-1} (a(s)) |y| \right| \big) = \varphi_{|x|}(s) + \varphi_{|y|}(s).
    \end{align*}
    
    To prove \ref{phiincreasing}, let $x \in S^+(\mathcal{M}) \cap L^{1}_{\mathcal{A},\text{Loc}}$, $s_1 \leq s_2$, and $p,q \in \mathcal{A}$ satisfying $\tau(p) = a(s_1)$ and $\tau(q) = a(s_2)$. Since $a$ is increasing, we have that $p \leq q$. Thus $\varphi_x(s_1) = \tau(|px|)$ and $\varphi_{x}(s_2) = \tau(|qx|)$. We will show that $|px| \leq |qx|$. By \cite[Theorem~1.7.3]{dodsbook} we have that $\mathfrak{D}(|px|) = \mathfrak{D}(px)$, $\mathfrak{D}(|qx|) = \mathfrak{D}(qx)$ and for each $\xi \in \mathfrak{D}(|px|) \cap \mathfrak{D}(|qx|)$ we have the equalities $\|px \xi\|_H = \||px|\xi\|_{H}$ and $\|qx \xi\|_H = \||qx|\xi\|_{H}$. Notice that $\mathfrak{D}(px) = \mathfrak{D}(x) = \mathfrak{D}(qx)$ and by orthogonality of $p$ and $q-p$ we get
    \begin{align*}
        \|qx \xi\|_H^2 &= \| px\xi + (q-p)x\xi \|^2_H = \|px \xi\|_{H}^2 + \|(q-p)x \xi\|_{H}^2 \geq \|px \xi\|_H^2.   
    \end{align*}
    By \cite[Proposition~2.2.24]{dodsbook} we have that $|px|^2 \leq |qx|^2$, an application of  \cite[Corollary~2.2.28]{dodsbook} shows that $|px| \leq |qx|$. Monotonicity of the trace shows that $\varphi_{x}(s_1) \leq \varphi_{x}(s_2)$ and proves that $\varphi_x$ is increasing.

    Since $\Gamma_{\mathcal{A}_M}$ is closed, its complement is a countable union of open intervals. It immediately follows that $a$ is constant on those intervals, thus $\varphi_x$ is also constant on those intervals and completes the proof of statement \ref{phiincreasing}. 

    To show statement \ref{phiapprox}, fix $x \in L^{1}_{\mathcal{A},\text{Loc}}$. Using the spectral measure $e^{|x|}$, define the bounded positive operators
    $$
x_n = \int\limits_{[0,n]} r \, d e^{|x|}(r), \quad \forall n \in \mathbb{N}^+.  
    $$
    Fix $p \in \mathcal{A}_M$, we will show that $|x_n p|$ is an increasing sequence bounded above by $||x|p|$.

    Fix $\xi \in \mathfrak{D}(xp)$, then $p(\xi) \in \mathfrak{D}(x)$. Using \cite[Theorem~1.5.7]{dodsbook} we have
    \begin{align*}
\| ||x|p|(\xi) \|^2_{H} &= \| |x|p(\xi) \|^2_{H} = \int\limits_{[0,\infty)} r^2 \, d e^{|x|}_{p \xi,p\xi} = \sup_{n} \int_{[0,n]} r^2 \, d e^{|x|}_{p \xi,p\xi} \\
&= \|x_n p \xi\|^2_H = \||x_n p| \xi\|^2_H.   
    \end{align*}
    The above inequality shows that $\||x_n p| \xi\|_H$ increases to $\|\left||x| p\right| \xi\|_H$. Since $x_n \in \mathcal{M}^+$, it is clear that $\mathfrak{D}(x p) \subseteq \mathfrak{D}(x_n p)$ for each $n$. By \cite[Proposition~2.2.24]{dodsbook} we have that $|xp_n|^2 \leq |x p_{n+1}|^2 \leq ||x|p|^2$, an application of  \cite[Corollary~2.2.28]{dodsbook} shows that $|xp_n| \leq |x p_{n+1}| \leq \left||x|p\right|$ for all $n \in \mathbb{N}^+$. 

    Since $\{|x_n p|\}$ and $\{ |x_n p|^2 \}$ are increasing sequences in $S(\tau)$ bounded above by $||x|p|, ||x|p|^2 \in S(\tau)$ respectively, by \cite[Proposition~2.3.10]{dodsbook} $|x_n p| \uparrow y \leq |xp|$ and $|x_n p|^2 \uparrow z \leq |xp|^2$  for some $y,z \in S(\tau)^+$. 

    The sequences $\{y - |x_np|\}$ and $\{z - |x_np|^2\}$ are positive, bounded above by $|xp|$ and $|xp|^2$ which are operators in $S_0(\tau)$ and the sequences satisfy $y - |x_np| \downarrow 0$ and $z - |x_np|^2 \downarrow 0$. By \cite[Theorem~2.6.3]{dodsbook} we have the convergence $|x_np| \to y$ and $|x_np|^2 \to z$ in the measure topology. An application of \cite[Theorem~2.8.7]{dodsbook} to the square root function shows that $|x_np| \to z^{\frac{1}{2}}$. By uniqueness of limits in the measure topology, we conclude that $y = z^{\frac{1}{2}}$ and $y^2 \leq |px|^2$.

    Let $\xi \in \mathfrak{D}(y)$, since $|x_n p|^2 \leq y^2$, we have that $\| x_np \xi \|_{H}^2 \leq \|y \xi\|^2_{H}$. Therefore
    \begin{align*}
        \int\limits_{[0,n]} r^2 \, de^{|x|}_{p \xi,p\xi} &\leq \|y \xi\|^2_{H} < \infty.
    \end{align*}
    Taking supremum we get $\|xp \xi\|_{H}^2 \leq \|y \xi\|^2_{H} < \infty$, therefore $p \xi \in \mathfrak{D}(x)$. This shows that $\mathfrak{D}(y) \subseteq \mathfrak{D}(|xp|)$ and also that $\||xp| \xi\|_{H} \leq \|y \xi\|_{H}$ for all $\xi \in \mathfrak{D}(y)$. By \cite[Proposition~2.2.24]{dodsbook} we have that $|xp|^2 \leq y^2$. Therefore $y^2 = ||x|p|^2$ and it follows that $y = ||x|p|$.

    Hence $|x_np| \uparrow ||x|p|$ for all $p \in \mathcal{A}_M$. By normality of the trace we get
    \begin{align*}
\varphi_{x_n}(s) &= \tau\big( \left|  \tau_{\mathcal{A}_M}^{-1} (a(s)) x_n \right| \big) = \tau\big( \left|  x_n \, \tau_{\mathcal{A}_M}^{-1} (a(s)) \right| \big) \uparrow \tau\big( \left|  |x| \, \tau_{\mathcal{A}_M}^{-1} (a(s)) \right| \big) \\
&= \varphi_{|x|}(s),
    \end{align*}
    for all $s \geq 0$. This completes the proof of \ref{phiapprox}.
    
\end{proof}

We now introduce the level function related to a measurable operator in $L^{1}_{\mathcal{A},\text{Loc}}$.

\begin{defn}\label{deflevel}
        Let $x \in L^{1}_{\mathcal{A},\text{Loc}} \cap S^{+}(\mathcal{M})$ and $\varphi_x$ the function defined in Proposition \ref{phiprop}. We say that a function $f:[0,\infty) \to [0,\infty)$ is a $\mathcal{A}$-level function of $x$ if:
        \begin{enumerate} [label=(\roman*)]
            \item \label{leveldec} The function $f$ is non-increasing $\lambda$-measurable, where $\lambda$ is the measure induced by the core $\mathcal{A}$.
            \item \label{levelmajorant} The inequality 
            $$
\varphi_{x}(s) \leq \int_{0}^s f
            $$
            holds for all $s > 0$.
            \item \label{levelleast} If $G$ is a concave majorant of $\varphi_x(s)$, then 
            $$
\int_{0}^s f \leq G(s),
            $$
            holds for all $s > 0$.
        \end{enumerate}
\end{defn}

We establish the existence and uniqueness of $\mathcal{A}$-level functions and relate them to a $\tau$-measurable operator.

\begin{thm}\label{existslevel}
    Let $\mathcal{M}$ be a von Neumann algebra equipped with a semifinite normal faithful trace $\tau$ and an ordered core $\mathcal{A}$. Let $x \in S(\mathcal{M})^+ \cap L^1_{\text{Loc},\mathcal{A}}$ such that the family 
    \begin{equation}\label{condnotempty}
\mathcal{F}_x = \left\{ g : g \text{ is concave and } \varphi_x(s) \leq g(s), \ \forall s> 0   \right\} 
\end{equation}
is not empty, and that $x$ satisfies
\begin{equation}\label{conddecr}
    \tau(|p_n x|) \downarrow 0, \quad \text{if } \{p_n\} \subseteq \mathcal{A} \quad \text{ and} \quad \tau(p_n) \downarrow 0.
\end{equation} Then, there exists a unique $\mathcal{A}$-level function of $x$ denoted $f_x$. Moreover, there exists a unique core decreasing operator $x^o$ such that $R(x^o) = f_x$ and $x^o = Q(f_x)$. Here $R$ and $Q$ are the transition maps from Propositions \ref{Rmap} and \ref{Qmap}.
\end{thm}
\begin{proof}
    Let $G(s) = \inf_{\mathcal{F}_x} g(s)$. The pointwise infimum of a family of concave functions is concave and, from the construction, it is clear that $\varphi_x(s) \leq G(s)$ for all $s > 0$. Since concave functions are absolutely continuous, there exists a decreasing function $f_x$ such that 
$$
\varphi_x(s) \leq G(s) = G(0^+) + \int_0^s f_x(s).
$$
We now show that $G(0^+) = 0$. By hypothesis $x$ satisfies (\ref{conddecr}), therefore $\varphi_x(0^+) = 0$. Suppose that $G(0^+) = \delta > 0$ and choose $s_0 > 0$ small enough such that $\varphi_x(s) \leq \delta/2$ for all $s \leq s_0$, and consider the function
$$
W(s) = \bigg(s \frac{G(s_0)-\delta/2}{s_0} + \delta/2 \bigg) \chi_{[0,s_0]}(s) + G(s) \chi_{(s_0,\infty)}(s).
$$
By construction, $W(s) \geq \delta/2 \geq \varphi_x(s)$ for all $s \leq s_0$ and $W(s) = G(s) \geq \varphi_x(s)$ for all $s > s_0$. Set $m = \frac{G(s_0)-\delta/2}{s_0}$ and notice that
$$
W(s) = \delta/2 + \int_0^s \big(m \chi_{[0,s_0]} + f_x \chi_{(s_0,\infty)}\big).
$$
To show that $W$ is a concave majorant of $\varphi_x$, it suffices to show that $m \geq \esssup_{(s_0,\infty)} f_x$. Set $r = \esssup_{(s_0,\infty)} f_x$ and suppose that $r > m$. Using the fact that $f_x$ is decreasing, we get
$$
G(s_0) = \delta + \int_{0}^{s_0} f_x \geq \delta + \int_{0}^{s_0} r = \delta  + rs_0 > \delta + m s_0 > \delta/2 + m s_0 = G(s_0),  
$$
therefore $G(s_0) > G(s_0)$, which is impossible. This shows that $W$ is a concave majorant of $\varphi_x$, and by minimality of $G$, we get that $G(s) \leq W(s)$ for all $s > 0$. Letting $s \downarrow 0$ we get $\delta \leq \delta/2$ which is impossible. Hence, $G(0^+) = 0$.

This shows that $f_x$ satisfies items \ref{levelmajorant} and \ref{levelleast} in Definition \ref{deflevel}. 

It remains to show that $f_x$ is $\lambda$-measurable. Since $f$ is decreasing, it suffices to show that $f$ is constant outside of $\Gamma_{\mathcal{A}_M}$. 

By assumption, $\Gamma_{\mathcal{A}_M}$ is closed in $[0,\infty)$; thus, its complement is a countable union of open intervals of the form $(c,d)$ or of the form $(d,\infty)$. 

We proceed to show that $f_x$ is constant on any bounded connected component $(c,d)$ of the complement of $\Gamma_{\mathcal{A}_M}$. Consider the function $$L(s) = G(s) \big(1-\chi_{(c,d)}(s)\big) + \Big(\frac{G(d) - G(c)}{d-c} (s-c) + G(c) \Big) \chi_{(c,d)}(s).$$

It is clear that $L$ is differentiable almost everywhere and its derivative is $$l(s) = f_x \Big( 1 - \chi_{(c,d)} \Big) + \frac{G(d) - G(c)}{d-c} \chi_{(c,d)}$$. Since $f_x$ is decreasing, it follows that 
$$
\essinf_{[0,c]} f_x \geq f_x(s) \geq  \esssup_{[d,\infty)} f_x,  \quad \forall s \in (c,d).
$$
Thus 
$$
\essinf_{[0,c]} f_x \geq \frac{1}{d-c} \int_{c}^d f_x(s) \geq  \esssup_{[d,\infty)} f_x,
$$
hence $\essinf_{[0,c]} f_x \geq l(s) \geq  \esssup_{[d,\infty)} f_x$ for all $s \in (c,d)$. It follows that $l$ is a decreasing function, therefore $L$ is concave. Notice that $L(c) = G(c)$, $L(d) = G(d)$, and $L$ is increasing, therefore 
$$
\varphi_x(s) = \varphi(c^+) \leq G(c) \leq L(s), \quad \forall s \in (c,d). 
$$
Thus $L$ is a concave majorant of $\varphi_x$. By minimality, we conclude that $G(s) \leq L(s)$. However, since $G$ is concave, for any $s \in (c,d)$ we must have that $G(s)$ lies above the convex combination of $G(c)$ and $G(d)$ associated with $s$, hence $G(s) \geq L(s)$ and we conclude that $G(s) = L(s)$. Therefore, $f_x$ is constant on $(c,d)$. 

Similarly, if $(d,\infty)$ is a connected component of the complement of $\Gamma_{\mathcal{A}_M}$ and $s > d$, then $\varphi_x(s) = \varphi_x(d)$. Therefore the function $s \mapsto \int_{0}^s f_x \chi_{[0,d]}$ is a concave majorant of $\varphi_x$. It follows that $G(s) \leq  \int_{0}^s f_x \chi_{[0,d]}$, with equality holding when $s \leq d$. We conclude that $f_x = 0$ for $s > d$ and prove that $f_x$ is constant on the complement of $\Gamma_{\mathcal{A}_M}$.

To show the uniqueness of $f_x$, suppose that $g_x$ is another $\mathcal{A}$-level function of $x$. By \ref{levelleast} in Definition \ref{deflevel} we have the inequalities
$$
\int_{0}^s f_x \leq  \int_{0}^s g_x, \quad \text{and} \quad  \int_{0}^s g_x \leq  \int_{0}^s f_x, \quad \forall s > 0
$$
It follows that $f_x = g_x$ almost everywhere.

The function $f_x$ is decreasing, so for any $s \in \Gamma_{\mathcal{A}_M} \setminus \{0\}$ the function $f_x \chi_{[s,\infty)}$ is bounded, thus $f_x$ is $\lambda$-measurable and by the $*$-isomorphism of measurable functions, we get that $x^o =Q(f_x)$ satisfies the rest of the statements. 
\end{proof}

As a consequence of the previous theorem, we can define level operators.
\begin{defn}
    Let $(\mathcal{M},\tau,\mathcal{A})$ be as in the previous theorem. For each $x \in L^{1}_{\mathcal{A},\text{Loc}}$ satisfying (\ref{conddecr}), its \textbf{level operator} is the operator $(|x|)^o$ constructed in Theorem \ref{existslevel}.
\end{defn}

The following statement shows that the collection of measurable operators satisfying (\ref{conddecr}) at least contains the von Neumann algebra $\mathcal{M}$ and the space $L^{1}(\tau)$.

\begin{prop}
    Let $x \in \mathcal{M}$ and $y \in L^{1}(\tau)$, then $\varphi_x(0^+) = 0 = \varphi_y(0+)$ and the collections $\mathcal{F}_x$ and $\mathcal{F}_y$ are not empty. In consequence, there exists unique level operators $x^o$ and $y^o$ in $\mathcal{M}_{\mathcal{A}}$
\end{prop}
\begin{proof}
     Let $p \in P(\mathcal{M})$, notice that for any $\xi \in H$ we have
    $$
\langle |xp|^2 \xi, \xi \rangle = \langle |xp| \xi, |xp| \xi \rangle = \|xp \xi\|^2_{H} \leq \|x\|^2_{B(H)} \|p \xi \|^2_{H} = \langle \|x\|^2_{B(H)} p \xi, \xi \rangle. 
    $$
    It follows that $|xp|^2 \leq \|x\|_{B(H)}^2 p$, therefore $|xp| \leq \|x\|_{B(H)} p$. 

Hence, for any $s > 0$ we have
\begin{align*}
\varphi_{x}(s) &= \tau\big( \left| \tau_{\mathcal{A}_M}^{-1} (a(s)) x \right| \big) = \tau\big( \left| x\ \tau_{\mathcal{A}_M}^{-1} (a(s))  \right| \big) \leq \|x\|_{B(H)} \tau\big( \ \tau_{\mathcal{A}_M}^{-1} (a(s))  \big) \\
&= \|x\|_{B(H)} a(s) \leq s \|x\|_{B(H)}. 
\end{align*}
The function $s \mapsto \|x\|_{B(H)} s$ is a concave majorant of $\varphi_x$, therefore, the collection
$$
\mathcal{F}_x = \left\{ g : g \text{ is concave and } \varphi_x(s) \leq g(s), \ \forall s> 0   \right\} 
$$
is not empty. Also letting $s \downarrow 0$ we get $\varphi(0^+) = \lim\limits_{s \to 0^+} \|x\|_{B(H)} s = 0$. Therefore, Theorem $\ref{existslevel}$ applies and shows that a unique level operator $x^o$ exists.

If $y \in \mathcal{M}$, by the argument used in the proof of Proposition \ref{phiprop}.\ref{phiincreasing}, we have that if $\{p_n\} \downarrow 0$, then $|p_n y| \downarrow 0$. Since multiplication is continuous in the measure topology and $y \in S(\tau)$, then $p_n y \to 0$ in the measure topology. Thus, $|p_n y| \to 0$ in measure. Notice that $|p_n y| \leq |y|$, therefore by the dominated convergence theorem \cite[Theorem~3.4.21]{dodsbook} we get $\tau(|p_n y|) \to 0$, hence $\varphi_y(0^+) = 0$. 

Clearly, $\varphi_y$ is dominated by the concave function $s \mapsto \|y\|_{L^{1}(\tau)}$, thus $\mathcal{F}_{y}$ is not empty. Another application of Theorem $\ref{existslevel}$ shows that a unique level operator $y^o$ exists and completes the proof.

\end{proof}
We now show that the level function extends the construction done in \cite[Section~7]{coredecreasing}.
\begin{thm}\label{levelfunctionextends}
    Let $x \in S(\mathcal{M}_{\mathcal{A}}) \cap L^{1}_{\text{Loc},\mathcal{A}}$ satisfying (\ref{conddecr}). Then $$x^o = Q\Big(\big(R(x)\big)^o\Big),$$ where $\big(R(x)\big)^o$ is the classic level function for $\lambda$-measurable functions on the half line. Here $Q$ and $R$ are the extended operators from Proposition \ref{QRmaps}.
\end{thm}
\begin{proof}
    Since $x$ and $p$ commute, we have $\tau(|xp|) = \tau(xp)$. Therefore, using Proposition \ref{QRmaps}.\ref{QRfunctor} we get
    $$
\varphi_x(\tau(p)) = \tau(|xp|)= \tau(xp) = \int_{[0,\tau(p)]} Rx \, d\lambda, \quad \forall p \in \mathcal{A}.
    $$
    Then,
    $$
\varphi_x(\tau(p)) \leq \int_{[0,\tau(p)]} (Rx)^o \, d\lambda.
    $$
    By minimality,
    $$
\int\limits_{[0,\tau(p)]} R(x^o) \, d\lambda \leq \int\limits_{[0,\tau(p)]} (Rx)^o \, d\lambda, \quad \forall p \in \mathcal{A} 
    $$
    Conversely, notice that $s \mapsto \int_{[0,s]} R(x^o) \, d\lambda$ is a $\lambda$-concave majorant of the function $s \mapsto \int_{[0,s]} Rx  d\lambda$, therefore by minimality we get 
    $$
\int\limits_{[0,\tau(p)]} (Rx)^o \, d\lambda \leq \int\limits_{[0,\tau(p)]} R(x^o) \, d\lambda, \quad \forall p \in \mathcal{A}.
    $$
    Therefore $(Rx)^o = R(x^o)$ $\lambda$-almost everywhere. An application of $Q$ yields $x^o = QR(x^o)$ and completes the proof. 
\end{proof}

With these tools in place, we can define noncommutative down spaces.

\begin{defn}\label{defdownnorm}
    Let $(\mathcal{M},\tau,\mathcal{A})$ be a von Neumann algebra equipped with a semifinite normal faithful trace $\tau$ and a $\sigma$-bounded full ordered core $\mathcal{A}$. Consider a strongly symmetric Banach $\mathcal{M}$-bimodule of $\tau$-measurable operators such that $\mathcal{A} \subseteq E \cap E^\times$. 
    
    For each $x \in S(\mathcal{M}) \cap L^{1}_{\text{Loc},\mathcal{A}}$, define 
    \begin{equation}\label{eqdownnorm}
        \|x\|_{E^o} = \|x^o\|_{E},
    \end{equation}
    whenever $x$ satisfies (\ref{conddecr}) and (\ref{condnotempty}), otherwise set $\|x\|_{E^o} = \infty$.

    The \textbf{down space of $E$} is the collection of measurable operators
    \begin{equation}\label{downspace}
E^o = \left\{ x \in S(\mathcal{M}) \cap L^{1}_{\text{Loc},\mathcal{A}}: \|x\|_{E^o} < \infty  \right\}.
    \end{equation}
 \end{defn}

 We give a representation formula for the down norm in terms of a fully symmetric Banach function space.
\begin{thm}\label{repdownnorm}
    Let $\mathcal{M,\tau,\mathcal{A}},E$ and $x$ be as in Definition \ref{defdownnorm}, $R$ is the map from Proposition \ref{Rmap} and $f_x$ the $\mathcal{A}$-level function of $x$. Then, the equality
    \begin{equation}\label{levelrearrangement}
        \int_{0}^s f_x = \int_{0}^s \mu_{\lambda}(R(x^o)), 
    \end{equation}
    holds for all $s>0$. Moreover, there exists a u.r.i Banach function space $X_E$ of Lebesgue measurable functions on $[0,\infty)$ such that $\|x\|_{E^o} = \|\mu_{\lambda}(R(x^o))\|_{X_E}$.
\end{thm}
\begin{proof}
Since $f_x$ is constant on the complement of $\Gamma_{\mathcal{A}_M}$, we have that $f_x \circ b = f_x$, where $b$ is the function from Proposition \ref{abmaps}. An application of Theorem \ref{QRmaps}.\ref{QRinverse} ,and formula (\ref{pushforward2}) yields
\begin{align*}
\int_{0}^s f_x &= \int_{0}^s f_x\circ b = \int\limits_{[0,s]} f_x \, d\lambda = \int\limits_{[0,s]} RQ f_x \, d\lambda = \int\limits_{[0,s]} R (x^o) \, d\lambda \\&= \int_{0}^s R(x^o) \circ b.
\end{align*}
The same proof as in \cite[Lemma~8.9]{coredecreasing} shows that $R(x^o) \circ b = \mu_\lambda(R(x^o))$. This completes the proof of (\ref{levelrearrangement}).

The condition that $\mathcal{A} \subseteq E \cap E^\times$ together with the condition that $\mathcal{A}$ is a full ordered core, forces that the carrier projections of $E$ and $E^\times$ are the identity operator. Therefore, an application of \cite[Proposition~5.1.6]{dodsbook} provides a fully symmetric space $X_{E}$ over $[0,\infty)$ such that $$\|x^o\|_{E} = \|x^o\|_{E^{\times \times}} = \|\mu_\tau(x^o)\|_{X_E}.$$
    Where $\mu_\tau$ is the nonincreasing rearrangement taken with respect to the trace $\tau$. By Theorem \ref{QRmaps}.\ref{Rrearrangement}, the map $R$ preserves the rearrangement of $x^o$, therefore 
    $$
\|x^o\|_{E} = \|\mu_{\lambda}(R(x^o))\|_{X_E},
    $$
    and completes the proof.
\end{proof}

With the previous representation formula, we are ready to prove that the down space is a normed $\mathbb{C}$-vector space of $S(\mathcal{M})$-measurable operators.

\begin{thm}\label{downisnormed}
    Let $\mathcal{M,\tau,\mathcal{A}},E$ and $x$ be as in Definition \ref{defdownnorm}. Let $x,y \in E^o$ and $\alpha \in \mathbb{C}$, then:
        \begin{enumerate}[label=(\roman*)]
     \item The equality $\|\alpha x\|_{E^o} = \left|\alpha\right| \|x\|_{E^o}$ holds.
     \item $\|x\|_{E^o} = 0$ if and only if $x = 0$.
     \item The triangle inequality $\|x+y\|_{E^o} \leq \|x\|_{E^o} + \|y\|_{E^o}$ holds.
\end{enumerate}
In consequence, $E^o$ is a normed complex vector space.
\end{thm}
\begin{proof}
By Proposition \ref{phiprop}.\ref{phihomogeneous}, $\varphi_{\alpha x} = |\alpha| \varphi_{x}$. If $G_{\alpha x}$ is the least concave majorant of $\varphi_{\alpha x}$ and $G_x$ is the least concave majorant of $\varphi_x$, it follows that $|\alpha| G_x$ is a concave majorant of $\varphi_{\alpha x}$ and that $\frac{1}{|\alpha|} G_{\alpha x}$ is a concave majorant of $\varphi_x$. Therefore,
$$
\int_{0}^s f_{\alpha x} \leq \int_{0}^s f_{x}, \quad \text{and} \int_{0}^s f_{x} \leq \int_{0}^s \frac{1}{|\alpha|} f_{\alpha x}, \quad \forall s>0.
$$
Using the fact that $Q f_x = x^o$, $Q f_{\alpha x} = (\alpha x)^o$, and formula \ref{levelrearrangement} we get
$$
\int_{0}^s \mu_\lambda \big(R((\alpha x)^o)\big) \leq \int_{0}^s |\alpha| \mu_{\lambda}\big(R(x^o)\big),$$
and  $$\int_{0}^s \mu_\lambda\big(R(x^o)\big) \leq \int_0^s \frac{1}{|\alpha|} \mu_\lambda\Big( R((\alpha x)^o) \Big)   
$$
hold for all $s > 0$. Let $E_X$ be the u.r.i function norm provided by Theorem \ref{repdownnorm}, from the above inequalities, we conclude that $$\|\mu_\lambda(R((\alpha x)^o)\|_{E_X} \leq \||\alpha| \mu_\lambda(R( x)^o)\|_{E_X}$$ and that $\|\mu_\lambda(R( x)^o)\|_{E_X} \leq \|\frac{1}{|\alpha|}\mu_\lambda(R((\alpha x)^o))\|_{E_X}$. Therefore, $\|\mu_\lambda(R((\alpha x)^o)\|_{E_X} = |\alpha|  \|\mu_\lambda(R( x)^o)\|_{E_X}$, and an application of Theorem \ref{repdownnorm} shows that $\|\alpha x\|_{E^o} = |\alpha|\|x\|_{E^o}$ and proves statement (i).

To prove statement (ii). It is clear that if $x = 0$, then $\varphi_x \equiv 0$ and we get $x^o = 0$, therefore $\|x\|_{E^o} = \|x^o\|_{E} = 0$. Conversely, suppose that $\|x\|_{E^o} = 0$. Then $\|x^o\|_{E} = 0$, hence $x^o = 0$ and $R(x^o) = 0$. By construction of $x^o$ we must have that
$$
\varphi_{x}(s) \leq \int_{0}^s \mu_\lambda(R(x^o)) = 0.
$$
Therefore, $\tau(|xp|) = 0$ for all $p \in \mathcal{A}$. Since $||x|p| = |xp|$ and the trace is faithful, we conlcude that $|x|p = 0$. Letting $p_n \uparrow 1$ yields $|x| = 0$ and completes the proof of (ii).

To prove (iii), from Proposition \ref{phiprop}.\ref{philinear} we have the inequality $\varphi_{x+y} \leq \varphi_x + \varphi_y$. Therefore the inequality
$$
\varphi_{x+y}(s) \leq \int_{0}^s \mu_\lambda\big( R(x^o)\big) + \int_{0}^s \mu_\lambda\big( R(y^o)\big) = \int_{0}^s \Big( \mu_\lambda\big( R(x^o)\big) + \mu_\lambda\big( R(y^o)\big) \Big)  
$$
holds for all $s > 0$.
Since $\mu_\lambda\big( R(x^o)\big) + \mu_\lambda\big( R(y^o)\big)$ is a decreasing function, its integral is a concave majorant of $\varphi_{x+y}$. Thus,
$$
\int_{0}^s \mu_\lambda\big( R((x+y)^o) \big) \leq \int_{0}^s \Big( \mu_\lambda\big( R(x^o)\big) + \mu_\lambda\big( R(y^o)\big) \Big), \quad \forall s>0.  
$$
Since $X_E$ is u.r.i, we get the inequalities 
\begin{align*}
  \|\mu_\lambda\big( R((x+y)^o) \big)\|_{E_X} &\leq \|\mu_\lambda\big( R(x^o)\big) + \mu_\lambda\big( R(y^o)\big)\|_{E_X}\\
  &\leq \|\mu_\lambda\big( R(x^o)\big)\|_{E_X} + \|\mu_\lambda\big( R(y^o)\big)\|_{E_X}.  
\end{align*}
Another application of Theorem \ref{repdownnorm} yields $\|x+y\|_{E^o} \leq \|x\|_{E^o} + \|y\|_{E^o}$ and completes the proof.
\end{proof}

In the following example, we compute explicitly the $\mathcal{A}$-level function of some operators, and we show that the equality $\|x^*\|_{E^o} = \|x\|_{E^o}$ may fail.

\begin{ex}
    Let $\mathcal{M},\tau,\mathcal{A}$ be from Examples \ref{exHilbert} and \ref{exHilbert2}. Consider the operator $x \in B(H)$ defined by
    $$
x(e_1) = e_1, \quad x(e_2) = 2e_1 + e_2, \quad x(e_k) = 0, \quad \forall k>2.
    $$
A direct computation shows that 
$$
|p_1 x|(e_1) = \frac{1}{\sqrt{5}} e_1 + \frac{2}{\sqrt{5}} e_2, \quad |p_1 x|(e_2) = \frac{2}{\sqrt{5}} e_1 + \frac{4}{\sqrt{5}} e_2, \quad |p_1x|(e_k) = 0,  
$$
for all $k>2$, and $|p_jx| = |x|$ for all $j > 1$, which is given by
$$
|x|(e_1) = \frac{1}{\sqrt{2}} e_1 + \frac{1}{\sqrt{2}} e_2, \quad |x|(e_2) = \frac{1}{\sqrt{2}} e_1 + \frac{3}{\sqrt{2}} e_2, \quad |x|(e_k) = 0, \forall k>2.  
$$
It follows that 
$$
\varphi_x = \sqrt{5} \chi_{[1,2)} + 2\sqrt{2} \chi_{[2,\infty)}.
$$
The least concave majorant of $\varphi_x$ is given by
$$
G_{x}(s) = \int_{0}^s \Big(\sqrt{5} \, \chi_{(0,1]}(t) + \big( 2\sqrt{2} - \sqrt{5}\big)\chi_{(1,2]}(t) \Big)  \, dt,
$$
thus $f_x = \sqrt{2} \chi_{(0,2]}$. Applying the operator $Q$ computed in Example \ref{exHilbert2} yields
$$
x^o(\xi)  = Q(f_x)(\xi) = \sqrt{5} \langle \xi, e_1 \rangle e_1 + \big(2\sqrt{2}-\sqrt{5}\big) \langle \xi, e_2 \rangle e_2, 
$$
and $\|x\|_{(L^{\infty}(\tau))^o} = \|x^o\|_{B(H)} = \sqrt{5}$.

Its adjoint is defined by
    $$
x^*(e_1) = e_1 + 2e_2, \quad x^*(e_2) = e_1, \quad x^*(e_k) = 0, \quad \forall k>2. 
    $$
Similarly, 
$$
|p_1 x^*|(e_1) = e_1 , \quad |p_1 x^*|(e_k) =  0, \quad \forall k>1,  
$$
and $|p_jx^*| = |x^*|$ for all $j > 1$, which is given by
$$
|x^*|(e_1) = \frac{3}{\sqrt{2}} e_1 + \frac{1}{\sqrt{2}} e_2, \quad |x|(e_2) = \frac{1}{\sqrt{2}} e_1 + \frac{1}{\sqrt{2}} e_2, \quad |x|(e_k) = 0, \forall k>2.  
$$
It follows that 
$$
\varphi_{x^*} = 1 \chi_{[1,2)} + 2\sqrt{2} \chi_{[2,\infty)}.
$$
The least concave majorant of $\varphi_x$ is given by
$$
G_{x^*}(s) = \int_{0}^s \sqrt{2} \chi_{(0,2]}(t) \, dt,
$$
thus $f_{x^*} = \sqrt{2} \chi_{(0,2]}$. Applying the operator $Q$ computed in Example \ref{exHilbert2} yields
$$
(x^*)^o(\xi)  = Q(f_{x^*})(\xi) = \sqrt{2} \langle \xi, e_1 \rangle e_1 + \sqrt{2} \langle \xi, e_2 \rangle e_2, 
$$
and $\|x^*\|_{(L^{\infty}(\tau))^o} = \|(x^*)^o\|_{B(H)} = \sqrt{2}$.

Therefore, for this operator, the strict inequality $\| x^*\|_{\mathcal{M}^o} < \| x\|_{\mathcal{M}^o}$ holds.

This example also shows that the hypothesis $x \in S(\mathcal{M}_{\mathcal{A}})$ in Theorem \ref{levelfunctionextends} is necessary. To see this, notice that 
$$
R(|x|) = \frac{\sqrt{2}}{2} \, \chi_{(0,1]} + \frac{3\sqrt{2}}{2} \, \chi_{(1,2]},
$$
thus $\big(R(|x|)\big)^o = \sqrt{2} \, \chi_{(0,2]}$, hence $Q \Big( \big(R(|x|)\big)^o \Big) = (x^*)^o \not= x^o$.

\end{ex}

We conclude with the following description of the down spaces for $L^{1}(\tau)$ and $\mathcal{M}$.

\begin{thm}
    Let $\mathcal{M}$ be a von Neumann algebra equipped with a semifinite normal faithful trace $\tau$ and a $\sigma$-bounded full ordered core $\mathcal{A}$. Then $(L^{1}(\tau))^o = L^{1}(\tau)$ with equality of norms and $\mathcal{M} \subseteq (\mathcal{M})^o$, where the inclusion can be proper.
\end{thm}
\begin{proof}
Let $x \in L^{1}(\tau)$, notice that 
$$
\varphi_x(\tau(p)) \leq \int\limits_{[0,\tau(p)]} R(x^o) \, d\lambda \leq \|x\|_{L^{1}(\tau)}, \quad \forall p \in \mathcal{A}.  
$$
Thus $\varphi_x(\tau(p)) \leq \tau(x^o p) \leq \tau(|x|)$. Choose a sequence $p_n \uparrow 1$ to get $\tau(x^o) = \tau(|x|)$. Thus $\|x\|_{(L^{1}(\tau))^o} = \|x\|_{L^{1}(\tau)}$.

If $y \in \mathcal{M}$, it was already shown that
$$
\int\limits_{[0,\tau(p)]} R(y^o) \, d\lambda \leq \int\limits_{[0,\tau(p)]} \|y\|_{B(H)} \, d\lambda, \quad \forall p \in \mathcal{A}.
$$
An application of the norm in $L^{\infty}_\lambda$ together with the fact that $R(y^o)$ is decreasing yields the inequality $\|R(y^o)\|_{L^{\infty}_\lambda} \leq \|y\|_{B(H)}$. Therefore, $y^o \leq \|y\|_{B(H)} 1_{B(H)}$. An application of the norm in $B(H)$ shows that $\|y\|_{\mathcal{M}^o} \leq \|y\|_{B(H)}$.

Theorem \ref{levelfunctionextends} and the fact that, for the classical down space $(L^{\infty})^o \not= L^{\infty}$, show that the inclusion can be proper.
    
\end{proof}

\printbibliography

\end{document}